\documentclass[12pt]{amsart}%
\usepackage{amsfonts}
\usepackage{epsfig}
\usepackage{graphicx}
\usepackage{amsmath}
\usepackage{amsfonts}
\usepackage{amssymb}
\usepackage{latexsym}
\usepackage{rotating}
\usepackage{pstricks, pst-node, pst-text, pst-3d}
\usepackage{amsbsy}
\usepackage{bm}
\usepackage{dsfont}
\usepackage{lscape}
\usepackage{verbatim}
\usepackage{slashbox}
\usepackage{pict2e}
\usepackage{marginnote}

\input{amssym.def}
\newtheorem{theorem}{Theorem}
\newtheorem{proposition}{Proposition}
\newtheorem{corollary}{Corollary}
\newtheorem{lemma}{Lemma}
\newtheorem{remark}{Remark}

\newtheorem{definition}{Definition}
\theoremstyle{remark}

\newcommand{\spn}{\text{\rm span}\,}

\newcommand{\Endo}{\text{\rm End}}
\newcommand{\rank}{\text{\rm rank\,}}
\newcommand{\id}{\text{\rm id}}

\input epsf
\begin{document}
\title[Free nilpotent and $H$-type Lie algebras...]{Free nilpotent and $H$-type Lie algebras. Combinatorial and orthogonal designs}
\author[K.~Furutani, I.~Markina,  and A.~Vasil'ev]{Kenro Furutani$^{\dag}$, Irina Markina$^{\ddag}$, and Alexander Vasil'ev$^{\ddag}$}

\thanks{The author$^{\dag}$ was supported by the Grand-in-aid for
  Scientific Research (C) No. 23540251 of JSPS (Japan Society for the
  Promotion of Science).  The authors$^{\ddag}$ have been  supported by the grants of the Norwegian Research Council \#204726/V30,
\#213440/BG, \#239033/F20; and EU FP7 IRSES program STREVCOMS, grant  no.
PIRSES-GA-2013-612669}
 

\address{K.~Furutani:  Department of Mathematics, Faculty of Science and Technology, Science University of Tokyo, 2641 Yamazaki, Noda, Chiba (278-8510), Japan}
\email{furutani\_kenro@ma.noda.tus.ac.jp}

\
\address{I.~Markina and A.~Vasil'ev: Department of Mathematics, University of Bergen, P.O.~Box 7803,
Bergen N-5020, Norway}
\email{irina.markina@uib.no}
\email{alexander.vasiliev@math.uib.no}

\begin{abstract}
The aim of our paper is to construct pseudo $H$-type algebras from the covering free nilpotent two-step Lie algebra as the quotient algebra by an ideal. We propose an explicit algorithm
of construction of such an ideal by making use of a non-degenerate scalar product. Moreover, as a byproduct result, we recover the existence of a rational structure on pseudo $H$-type algebras,
which implies the existence of lattices on the corresponding pseudo $H$-type Lie groups.
Our approach  substantially uses combinatorics and reveals the interplay of   pseudo $H$-type algebras with combinatorial and orthogonal designs.
One of the key tools is the family of Hurwitz-Radon orthogonal matrices.
\end{abstract}
\maketitle

\section{Introduction}

Any nilpotent Lie algebra is known \cite{Sato} to be obtained as a coset of a free nilpotent algebra by a quotient ideal. However, it is not clear how this ideal can be recovered according
to the properties of the given nilpotent algebra. We propose an explicit construction of the ideal and of the quotient map by making use of a non-degenerate scalar product to obtain
pseudo $H$-type algebras, i.e., the graded two-step nilpotent Lie algebras introduced in \cite{Ciatti, GKM, Kaplan} and intimately related to representations of Clifford algebras and
compositions of quadratic forms. Moreover, as a byproduct result, we recover the existence of a rational structure on pseudo $H$-type algebras,
which implies the existence of lattices on the corresponding pseudo $H$-type Lie groups, giving a new and essentially different proof of the corresponding result in~\cite{FurutaniMarkina}, see also \cite{Eber04, CrDod}.
 Mal'cev~\cite{Malc} introduced a {\it nilmanifold} in 1951 as a compact manifold with a connected, simply connected nilpotent Lie group acting on it.  A nilmanifold diffeomorphic  to the quotient space $L\backslash G$ of a nilpotent Lie group $G$ where $L$ is a uniform lattice. They represent an example of a  homogeneous space  which plays an important role in geometry and harmonic analysis, as well as in arithmetic combinatorics and ergodic theory more recently. Mal'cev's criterion \cite{Malc} guarantees the existence of  a uniform lattice $L\subset G$ provided rational structural constants
for the nilpotent Lie algebra $\mathfrak g$ of the Lie group $G$. So pseudo $H$-type Lie groups with lattices provide a broad class of nilmanifolds which can be studied and classified.

The proposed approach is expected to contribute to a classification of $H$-type and pseudo $H$-type algebras, in particular
to a description of the classes of lattices that admit every pseudo $H$-type group.  To our knowledge, the complete classification 
  of lattices is known only in the case of Heisenberg groups (\cite{GW}
  and also see \cite{TY}). This would also lead to a possible classification of non-diffeomorphic nilmanifold. However, these topics exceed the scope of this paper and are subject of future research.
The study of infinitesimal symmetries of pseudo $H$-type groups and the comparison with the Tanaka prolongations of simple algebras factorised by parabolic subgroups
are closely linked. We hope that our approach will allow us to contribute to the beautiful theory of simple groups. 

At the end we show relations with combinatorics and orthogonal designs   using our construction for pseudo $H$-type groups instead of the classical
application of Clifford algebras. In particular, we apply our results to square semiregular 1-factorization of a complete graph $K_{2n}$ and we fix a Geramita--Seberry Wallis problem on a maximal number of variables in an amicable orthogonal design \cite[Problem 5.17]{Geramita79}, where 
the orthogonality is understood with respect of a non-degenerate metric with the neutral signature.

The structure of the paper is as follows. 
In Section~2, we provide an overview of one of the main tools, the Hurwitz-Radon family of orthogonal matrices, and necessary background from graph theory. As a motivation, we attract reader's attention to interrelations of our results and methods with harmonic analysis. We also define uniform lattices on Lie groups and describe relations between them and the presence of rational structural constants for the corresponding Lie algebras.
We give precise definitions and preliminaries  in Section 3. We prove the main results on construction of quotient maps and the  existence of rational structural constants in Section 4. Section~4.3 is dedicated to the classical case of $H$-type algebras with a positive definite metric and Section~4.4 generalizes it to pseudo $H$-type algebras. Section 5  gives connections
of the main results to combinatorial and orthogonal designs.
\medskip

\noindent
{\bf Acknowledgement.} The work on this paper  started at the Mittag-Leffler institute (Sweden) in the fall 2012, the program `Hamiltonians in Magnetic Fields',  and continued at the
Science University of Tokyo (Japan) in November-December 2012. The authors express their gratitude to all the staff of these institutions for their generous support. We are thankful to professor Anthony V.~Geramita and to the anonymous referee who drew our attention to the excellent monograph~\cite{Geramita79}, which shortened several parts of the paper and provided a link to  problems in orthogonal designs.

\section{Interrelations and motivations} 

The problem of composition of positive definite quadratic forms was proposed by Adolf  Hurwitz \cite{Hurwitz2} as early as in 1898. 
He \cite{Hurwitz}  and Johann Radon \cite{Radon} reformulated this problem  in 1922 to the problem of existence of a family of orthogonal $n\times n$--matrices $E_1,\dots, E_s$ satisfying 
the conditions
\[
E_j^2=-I,\quad E_iE_j+E_jE_i=0, \quad \text{for $j=1,\dots, s$ and $i\neq j$,} 
\]
where $I$ stands for the identity $(n\times n)$--matrix. This family is now known as the {\it Hurwitz-Radon (HR) family}.
One of their main results states that the maximal number of the matrices in the HR family is $\rho(n)-1$, where $\rho(n)$ is the Hurwitz-Radon function defined as  $\rho(n):=8\alpha+2^\beta$, where $n$ is uniquely represented by $n=u2^{4\alpha+\beta}$  with $u$  odd, $\beta=0,1,2$ or $3$. Later in 1943, Eckmann \cite{Eckmann} simplified the Radon's proof  using representation theory and group characters. The maximal number is achievable even by integer matrices, see \cite{GP}. 

The Hurwitz-Radon result is used in a bunch of pure and applied problems. In particular, the composition of quadratic forms exists if and only if there exists a representation of the Clifford algebras~\cite{Chevalley}, such that the representation space admits a quadratic form with some special property. It is used in the problem of existence of linearly independent vector fields  constructed on a sphere in $n$-dimensional Euclidean space~\cite{Adams}, in quantum mechanics as well as in in electronics \cite{Citko}, in space-time block coding (STBC) and orthogonal design~\cite{TJC}, in signal processing and computer vision~\cite{Jak}.

An important development in quantum computing, STBC and wireless communication happened after Wolfe~\cite{Wolfe} generalized Hurwitz-Radon result to non-positive definite quadratic forms and  matrices from an HR($s,t$) family such that $s$ first matrices satisfy the condition $E_j^2=-I$ but $t$ following matrices satisfy the condition $E_j^2=I$.
In our paper we essentially use this generalization as one of the main tools. 

Another group of problems related to our presentation is of graph theory and combinatorics nature.  It refers to a relationship between three concepts. The first one is {\it Hamiltonian graphs}, i.e.,
graphs containing a spanning cycle. The second one is a {\it square 1-factorization of a graph}.  A  1-factorization of a graph is a decomposition of all the edges of the graph into 1-factors, the sets of $k$ independent edges (without common vertices).   A graph admits a square 1-factorization if  the union of any two distinct 1-factors is a disjoint union of 4-cycles. The third concept is that of a {\it vertex transitive graph}, i.e., a graph such that the automorphism group acts transitively upon the vertices.  If a connected graph has a square 1-factorization, then it is vertex transitive and Hamiltonian.
 Ihrig~\cite{Ihrig} proved that  a graph admits a square 1-factorization if and only if it is a Cayley graph with the group  of automorphisms $(\mathbb Z_2)^n$ for some $n$. Since Cayley graphs of abelian groups are known to be Hamiltonian, graphs that admit square 1-factorization are Hamiltonian Cayley. Kobayashi and Nakamura~\cite{KN} proved that a complete graph $K_{2k}$ admits a square 1-factorization if and only if $k=2^n$, $n\geq 1$. As a side result, we generalize this  stating that there exists a 1-factorization of a complete graph $K_{2k}$ such that at least  $\rho(2k)-1$  one-factors satisfy the condition of square 1-factorization.
Relation with combinatorics  and space-time block coding are discussed in the last two sections of the paper.

Another perspective is coming from Fourier analysis  and its multivariate analogue on the torus $\mathbb T^n:=\mathbb Z^n\backslash\mathbb R^n$.
If $(M,g)$ is a smooth compact Riemannian manifold, then the Laplace operator $\Delta$ defined by $\Delta f=-\text{div}(\text{grad}(f))$, $f\in L^2(M,g)$, possesses a discrete spectrum $0=\lambda_0\leq\lambda_1\leq\dots\to\infty$. This simple fact naturally implies many interesting
problems related by a general question: how much  of the geometry of the manifold is determined by this spectrum. One of the seminal
papers in this direction appeared in 1966 by Mark Kac~\cite{Kac}. He asked whether the spectrum of the Laplacian on a compact planar domain with a boundary, acting on smooth functions vanishing on the boundary, determined its shape. This question was negatively answered~\cite{Gordon} in 1992.
Jeff Cheeger established the existence of a discrete spectrum  to Riemannian manifolds with conical singularities in a series of papers
about 1980, see, e.g., \cite{Cheeger}, which was further extended \cite{Muller} to the Riemannian manifolds with cusps. Another type of singular manifolds is provided by sub-Riemannian geometry. An interpretation of sub-Riemannian geometry can by thought of as follows. Take a Riemannian manifold $(M,g)$ and an orthonormal frame $(X_1,\dots, X_n)$ with respect to $g$. Define a family of Riemannian metrics
$g_{\varepsilon}$, $\varepsilon>0$, in $M$ by requiring that $(X_1,\dots, X_m,\varepsilon X_{m+1},\dots, \varepsilon X_n)$ is an orthonormal frame. Then the Gromov-Hausdorff limit of $(M,g_{\varepsilon})$ as $\varepsilon\to 0$ is a sub-Riemannian manifold (the inverse statement, in general, is not true).
 Equivalently it can be understood that, the Riemannian fiber metric $g_0$ is given only on a sub-bundle in $H\subset TM$, and then the triplet $(M,H,g_0)$ is called a sub-Riemannian manifold. Correspondingly, the Laplacian operator $\Delta$ can be changed to the sub-Laplacian one $\Delta_{SR}$, where the divergence and gradient are properly defined.
There exists a large amount of literature developing sub-Riemannian 
geometry. Typical general references are \cite{Mon, Str1, Str2}.  The simplest example of a sub-Riemannian manifold is the Heisenberg group $\mathbb H^1$ which is topologically $\mathbb R^3$, and the distribution is given by 
$$
H=\mbox{ker}\,{\omega}=\{(v_1,v_2,v_3),\,\,v_3-\frac{1}{2}(xv_2-yv_1)=0\}\subset\mathbb R^3.
$$
Equivalently, one defines $H_p=$span$(X_p,Y_p)$, $p\in \mathbb H^1$, where
\[
X_p=(\partial_x-\frac{1}{2}y\partial_z),\quad Y_p= (\partial_y+\frac{1}{2}x\partial_z),\quad p=(x,y,z),
\]
and $H$ is a subbundle of the tangent bundle $T\mathbb R^3$. We have $[X,Y]=Z=\partial_z$, and span$(X,Y,[X,Y])_p=\mathbb R^3$, a condition called bracket generating or H\"ormander, which guarantees the sub-Laplacian operator $\Delta_{SR}=-\frac{1}{2}(X^2+Y^2)$ to be hypoelliptic.

Next we discuss nilmanifolds which are left cosets $L\backslash G$ of a nilpotent Lie group $G$ where $L$ is a uniform lattice, see Definition~\ref{NM}, Section~\ref{SNM}.

Let $\Delta_{SR}^G$ be the  left-invariant sub-Laplacian on a group $G$ and $L\backslash G$ be a nilmanifold. The quotient map $G\to L\backslash G$ defines a
Grusin-type operator  $\Delta_{SR}^{L\backslash G}$ on the nilmanifold $L\backslash G$. The operator $\Delta_{SR}^{L\backslash G}$ has  a discrete spectrum $\Lambda\subset [0,\infty)$ analogously to the
case of $\mathbb T^n:=\mathbb Z^n\backslash\mathbb R^n$. Spectral geometry of such kind of nilmanifolds and operators was developed in \cite{BF-1,BF-2,BFI-2}. 


Mal'cev criterion \cite{Malc} guarantees the existence of  a uniform lattice $L\subset G$ provided rational structural constants
for the nilpotent Lie algebra $\mathfrak g$ of the Lie group $G$. In this connection, the study of the existence of the rational structural constants for $H$-type algebras is  important
for developing spectral theory of Grusin-type operators on nilmanifolds generated by $H$-type groups. 
The $H$-type Lie algebras were
introduced by Kaplan in~\cite{Kaplan} and were widely studied, see,
for instance~\cite{CCM,ChMar, Cowling91, Kap2,Kor}.
In works~\cite{Ciatti,CP,GKM} the analogues of the classical $H$-type Lie 
algebras, pseudo $H$-type algebras, where the positive definite metrics were replaced  by non-degenrate  metrics with arbitrary signatures, were introduced and studied. 
In particular, the existence of rational structural constants in pseudo $H$-type algebras was proved in~\cite{FurutaniMarkina}.

\section{Preliminaries}

\subsection{Clifford algebras and their representations}


Let $V$ be a real vector space endowed with a non-degenerate quadratic form $Q(v)$, $v\in V$, which defines a symmetric bilinear form $\mathbf{q}(u,v)=\frac{1}{2}(Q(u+v)-Q(u)-Q(v))$ by polarization. Here and further on by saying {\it scalar product} we mean
non-degenerate symmetric real bilinear form and by {\it inner product} a positive definite one. 
 The {\it Clifford algebra} $Cl(V,\mathbf q)$, named after the English geometer  William Kingdon Clifford~\cite{Cliff}, is an associative unital algebra freely generated by $V$ modulo the relations
\[
v^2=-Q(v)\mathds{1}=-\mathbf{q}(v,v)\mathds{1}\quad\text{for all $v\in V$ or } uv+vu=-2\mathbf{q}(u,v)\mathds{1}\quad\text{for all $u,v\in V$}.
\]
For an introductory text, one may look at~\cite{Garling}. Every non-degenerate quadratic form on the $n$-dimensional vector space $V$ is equivalent to the standard diagonal form
\[
Q_{p,q}(v)=v_1^2+v_2^2+\dots +v_p^2-v_{p+1}^2-\dots-v_{p+q}^2, \quad n=p+q.
\]
Using isomorphism $(V,Q)\simeq \mathbb R^{p,q}=(\mathbb R^{p+q},Q_{p,q})$ we will write $Cl(V,\mathbf q)=Cl_{p,q}$.  Starting with an orthonormal basis $\ell_1,\dots, \ell_n$ in $\mathbb R^{p,q}$ one defines 
a  basis of $Cl_{p,q}$ by the sequence
\[
1,\dots, (\ell_{k_1}\cdot\ldots\cdot \ell_{k_j}),\,\,1\leq k_1<k_2<\dots<k_j\leq n,\,\,j=1,2,\dots, n.
\]
It follows that the dimension of $Cl_{p,q}$ is $2^n$, $n=p+q$. 
We have the following algebra isomorphisms: $Cl_{0,0}\simeq\mathbb R$, $Cl_{1,0}\simeq\mathbb C$, $Cl_{0,1}\simeq\mathbb R^2$ where the product in $\mathbb R^2$ is defined componentwise. Furthermore, $Cl_{2,0}\simeq\mathbb H$, where $\mathbb H$ denotes quaternions, and $Cl_{1,1}\simeq Cl_{0,2}\simeq \mathbb R(2)$, where $\mathbb R(2)$ denotes $2\times 2$ matrices with real entries. 

Clifford algebras enjoy the isomorphisms 
\[
Cl_{p+1,q+1}=Cl_{p,q}\otimes\mathbb R(2),\quad Cl_{p+4,q}=Cl_{p,q+4},
\]
and the Cartan-Bott periodicity theorem \cite{Atiyah, Chern} implies that
\[
Cl_{p+8,q}=Cl_{p+4,q+4}=Cl_{p,q+8}=Cl_{p,q}\otimes\mathbb R(16).
\]
If the signature satisfies $p-q=1$(mod 4), then $Cl_{p+k,q}=Cl_{p,q+k}$.

A {\it Clifford  module} for $Cl_{p,q}$ is a representation of a Clifford algebra given by a finite-dimensional real space $U$ and a linear map $\rho\colon Cl_{p,q}\to {\rm End}(U)$, satisfying the Clifford relation $\rho^2(v)=-\mathbf{q}(v,v){\id}_{U}$ or
$\rho(u)\rho(v)+\rho(v)\rho(u)=-2\mathbf{q}(u,v){\id}_{U}$ for $u,v\in\mathbb R^{p,q}$. An abstract theory of Clifford modules    was founded in \cite{Atiyah}.  


\subsection{Free nilpotent Lie algebras}
Let $N$ stand for a {\it free nilpotent Lie algebra} defined following, e.g.,  \cite{Jacobson, Sato}. Given a real vector space $U$  of dimension $m$ with a basis $\{e_1\dots,e_m\}$, called {\it generators}, we construct the space
\[
\mathcal U= U\oplus (U\otimes U)\oplus\dots \oplus (U\otimes\dots\otimes U).
\]
We introduce a distributive, non-associative and non-commutative operation $\times\colon$  $\mathcal U\times \mathcal U\to \mathcal U$  iteratively as follows.
If $e_k\in U$, then
\[
e_i\times (e_{j_1}\otimes \cdots \otimes e_{j_k})=
\begin{cases}
e_i\otimes e_{j_1}\otimes \cdots \otimes e_{j_k},\quad & \text{if $k\leq n-1$},\\
0,\quad &  \text{if $k=n$}.
\end{cases}
\]
Next iteratively, if $e\in U$ and $a\times b$ is already defined for a fixed $a\in \mathcal U$ and for every $b\in \mathcal U$, then
\[
(e\otimes a)\times b=e\otimes (a\times b)-a\times (e\otimes b). 
\]
Let us denote by $\mathcal W$ the left ideal of the elements represented as $v\times v$, for $v\in \mathcal U$.
Let $N=\mathcal U/\mathcal W$. Then $N$ carries the structure of a nilpotent Lie algebra with the Lie product 
\[
[ \langle a\rangle ,  \langle b\rangle  ]:= \langle a\times b \rangle,
\]
where $ \langle a \rangle$ denotes the equivalence class with the representative $a\in \mathcal U$.  The Jacobi identity
holds and the nilpotentness trivially follows from $N^{n+1}=0$. Abusing notations we shall write simply $a$ instead of  $\langle a\rangle$ in what follows. The algebra $N$ is the free nilpotent Lie algebra of length $n$ with $m$ generators.

\begin{theorem}{\rm \cite{Sato}}  Let $\mathfrak{n}$ be a nilpotent Lie algebra of length $n$ generated by $m$ linearly independent elements.
Then there exists an ideal $A$ of $N$ such that $\mathfrak{n}\simeq  N/A$.
\end{theorem}

We denote by $\pi$ the projection $\pi\colon N\to \mathfrak{n}$. In particular, we are interested in free nilpotent Lie algebras of length 2.

\subsection{Pseudo $H$-type algebra}\label{eq:J_orth}

It is convenient for us to split the definition of
a pseudo $H$-type algebra in two parts.


\begin{definition}\label{metricLA}
We say that a Lie algebra $\mathfrak{n}\equiv(\mathfrak{n},[\cdot\,,\cdot],(\cdot\,,\cdot))$ is a two-step nilpotent metric Lie algebra if it satisfies the following properties:
\begin{itemize}
\item $[[\mathfrak n,\mathfrak n],\mathfrak n]=\{0\}$;
\item the scalar  product $(\cdot\,,\cdot)$ is non-degenerate;
\item $\mathfrak{n}=\mathfrak{h}\oplus_{\perp}\mathfrak{z}$ with respect to $(\cdot\,,\cdot)$, where $\mathfrak{z}$ is the center of $\mathfrak{n}$;
\item the restriction $(\cdot\,,\cdot)_{\mathfrak{z}}$ of $(\cdot\,,\cdot)$ to $\mathfrak{z}$ is non-degenerate.
\end{itemize}
\end{definition}

\begin{definition}\label{pseudoLA}
A two-step nilpotent metric Lie algebra  $\mathfrak{n}$ is called pseudo $H$-type algebra if the operator $J:\,\mathfrak{z}\times\mathfrak{h}\to \mathfrak{h}$ 
 defined by
 \begin{equation}\label{p1}
(J_z u,v)_{\mathfrak{h}}=(z, [u,v])_{\mathfrak{z}},
\end{equation}
satisfies the following orthogonality condition  
\begin{equation}\label{p2}
(J_zu,J_zv)_{\mathfrak{h}}=(z,z)_{\mathfrak{z}}(u,v)_{\mathfrak{h}}.
\end{equation}
\end{definition}

Given a metric two-step nilpotent Lie algebra $\mathfrak{n}$, we call the operator $J$ defined by \eqref{p1} satisfying \eqref{p2}, a pseudo {\it $H$-type structure} or simply an $H$-type structure in the case of a positive definite metric.

This definition for an inner product  (positive definite), is equivalent to that found in \cite{Kaplan}, and for a non-degenerate scalar  product, in \cite{Ciatti, GKM}. The definition~\eqref{p1} immediately implies the following properties of the operator $J$.
\begin{itemize}
\item The operator $J:\,\mathfrak{z}\times\mathfrak{h}\to \mathfrak{h}$  is bilinear. 
\item The operator
 $J_z\in\Endo(\mathfrak{h})$ for any fixed $z\in\mathfrak{z}$ is
 skew symmetric 
\begin{equation}\label{p4}
(J_z u, v)_{\mathfrak{h}}=-(u,J_z v)_{\mathfrak{h}},
\end{equation}
and
\begin{equation}\label{p3}
J^2_z=-(z,z)_{\mathfrak{z}}\id_{\mathfrak{h}};
\end{equation}
\item  The operator $J_z\colon \mathfrak{h}\to \mathfrak{h}$ is an isometry for all $z\in\mathfrak{z}$ with $(z,z)_{\mathfrak{z}}=1$;
\item  The operator $J_z\colon \mathfrak{h}\to \mathfrak{h}$ is an anti-isometry for all $z\in\mathfrak{z}$ with $(z,z)_{\mathfrak{z}}=-1$.
\end{itemize}

\begin{remark}\label{rem1}
Any two of \eqref{p2}, \eqref{p4}, and \eqref{p3} imply the third property.
\end{remark}

\begin{proposition}\cite[Proposition 2.2]{Ciatti}\label{pr4}
Let $\mathfrak{n}$ be a pseudo $H$-type algebra with a non-degenrate  scalar product $(\cdot\, ,\cdot)$. If the scalar product  $(\cdot\, ,\cdot)_{\mathfrak z}$ is non-positive definite, then necessarily, the signature of the scalar product $(\cdot\, ,\cdot)_{\mathfrak h}$ is neutral.
\end{proposition}

\begin{proposition}\label{propOrt}
Let $\mathfrak{n}$ be  a two-step nilpotent metric Lie algebra, and let $z_1,\dots, z_p$ be an orthogonal non-null basis of $\mathfrak{z}$. Let the operator $J$ be defined by \eqref{p1} on the basis vectors of $\mathfrak{z}$. The operator $J_z$, $z\in\mathfrak{z}$, satisfies condition \eqref{p2}, or equivalently,  $\mathfrak{n}$ is a pseudo $H$-type algebra,  if and only if, $J^2_{z_k}=-(z_k,z_k)_{\mathfrak{z}}\id_{\mathfrak{h}}$ and $J_{z_i}J_{z_j}=-J_{z_j}J_{z_i}$. 
\end{proposition}
\begin{proof}
The necessary condition is trivial. Let us focus ourselves on the sufficient part.
Indeed, the definition \eqref{p1} of the operator $J$ can be extended from the basis of $\mathfrak{z}$ to the whole $\mathfrak{z}$ by linearity, and it implies the skewsymmetry $(J_z u,v)_{\mathfrak{h}}=-(u,J_zv)_{\mathfrak{h}}$ and linearity with respect to $z$ and $u$. Therefore,
\[
(J^2_{\alpha z_1+\beta z_2}u,v)_{\mathfrak{h}}=-(J_{\alpha z_1+\beta z_2}u,J_{\alpha z_1+\beta z_2}v)_{\mathfrak{h}}
=-\alpha^2(J_{z_1}u,J_{z_1}v)_{\mathfrak{h}}-\beta^2(J_{z_2}u,J_{z_2}v)_{\mathfrak{h}}\]
\[=\alpha^2(J^2_{z_1}u,v)_{\mathfrak{h}}+\beta^2(J^2_{z_2}u,v)_{\mathfrak{h}}=-\alpha^2(z_1,z_1)_{\mathfrak{z}}(u,v)_{\mathfrak{h}}-\beta^2(z_2,z_2)_{\mathfrak{z}}(u,v)_{\mathfrak{h}}\]
\[=-(\alpha z_1+\beta z_2,\alpha z_1+\beta z_2)_{\mathfrak{z}}(u,v)_{\mathfrak{h}},
\]
for any $u$ and $v$ from $\mathfrak{h}$ and for any basis vectors $z_1$ and $z_2$ from $\mathfrak{z}$. Hence, the equality $J^2_{z}=-(z,z)_{\mathfrak{z}}\id_{\mathfrak{h}}$ is true for any vector $z\in\mathfrak{z}$ and the orthogonality condition  follows from~Remark~\ref{rem1}.
\end{proof}

Let us formulate how the pseudo $H$-type algebras are related to representations of Clifford algebras. Given a pseudo $H$-type algebra $\mathfrak{n}=(\mathfrak{z}\oplus_{\bot}\mathfrak{h}, [.\,,.])$, the operator $J_z$ is defined by~\eqref{p1} satisfying~\eqref{p3} for every $z\in\mathfrak{z}$, and therefore, it defines a representation $J\colon Cl(\mathfrak{z},\mathbf{q})\to\Endo(\mathfrak{h})$ over the space $\mathfrak{h}$. Here the quadratic form $\bf q$ is defined by the scalar product $(.\,,.)_{\mathfrak{z}}$. Moreover in this case, the scalar product $(.\,,.)_{\mathfrak{h}}$ on the representation space is such that the operator $J_z$ is skew symmetric with respect to this scalar product, see~\eqref{p4}. Following~\cite{Ciatti} we call such a pair $(h, (.\,,.)_{\mathfrak{h}})$ an {\it admissible representation}. Now let us assume that a representation $\rho\colon Cl(Z,\mathbf{q})\to\Endo(V)$ of the Clifford algebra $Cl(Z,\bf q)$ is given and let us suppose also that the representation space $V$ admits a scalar product $(.\,,.)_V$ such that the  restriction $J=\rho\vert_Z$ of $\rho$ to $Z$, is skew symmetric with respect to $(.\,,.)_V$, in other words $(V,(.\,,.)_V)$ is an admissible representation. Then we define the Lie brackets $[.\,,.]$ by formula~\eqref{p1}, where the scalar product $(.\,,.)_Z$ is the polarizaition of the quadratic form $\bf q$. The Lie algebra $\mathcal N=(Z\oplus V,[.\,,.])$ will be a pseudo $H$-type Lie algebra with the center $Z$ and with the orthogonal decomposition $Z\oplus V$ with respect to the scalar product $(.\,,.)=(.\,,.)_Z+(.\,,.)_V$.


\subsection{Lattices and nilmanifolds}\label{SNM}

Let $G$ be a nilpotent Lie group, and let  $\mathfrak{g}$ be its Lie algebra. We denote the
Lie exponent and logarithm by $\exp\colon \mathfrak g\to G$ and $\log\colon G\to \mathfrak g$.

\begin{definition}\label{NM}
A subgroup $L$ of $G$ is called a lattice if $L$ is discrete and the left co-set $L\backslash G$ possesses a finite  measure  which is invariant under the action of $G$, and inherited from the Haar measure on $G$. The lattice is called uniform (or co-compact) if  $L\backslash G$ is compact.
The   space $L\backslash G$ is called  compact nilmanifold.
\end{definition}

\begin{theorem}[Mal'cev criterion \cite{Malc}]\label{MC} The group $G$ admits a lattice $K$ if and only if  the Lie algebra $\mathfrak g$ admits a basis $\mathcal B=\{b_1,\dots, b_n\}$
with rational structural constants $[b_i,b_j]=\sum_{k=1}^{n}c_{ij}^k b_k$, $c_{ij}^k\in\mathbb Q$.
\end{theorem}

Given a lattice $K$, one can construct the corresponding basis $\mathcal B$ as follows. Set $\mathfrak g_{\mathbb Q}=\spn_{\mathbb Q}\log K$, which is a Lie algebra over the field $\mathbb Q$. Denote by $\mathcal B_{\mathbb Q}$ a $\mathbb Q$-basis in $\mathfrak g_{\mathbb Q}$. Then it is also an $\mathbb R$-basis $\mathcal B$ in $\mathfrak g$.

Reciprocally, given a basis $\mathcal B$ defined as in Theorem~\ref{MC}, let $\Lambda$ be a vector lattice in  $\mathfrak g$, such that $\Lambda\subset \spn_{\mathbb Q}\mathcal B$. Then the lattice $K$ is generated by elements $\exp \Lambda$ and $\spn_{\mathbb Q}(\log K)=\spn_{\mathbb Q}\mathcal B$.

Let $K_1$ and $K_2$ be two lattices in $G$. Then $\spn_{\mathbb Q}(\log K_1)=\spn_{\mathbb Q}(\log K_2)$, if and only if,  $(K_1\cap K_2)\backslash K_1$ and $(K_1\cap K_2)\backslash K_1$ are finite.

One of the aims of our paper is to prove that pseudo $H$-type groups admit  lattices, or equivalently, the  corresponding pseudo $H$-type algebras admit a basis with rational structural constants.


\section{$H$-type algebras from free algebras}

The main idea  is to construct the pseudo $H$-type algebra  by means of  the free algebra step by step restricting the latter by satisfying necessary conditions and enriching it with additional structures. 

\subsection{Construction}
First we assign the vector space $U$ to be the complementary space to the center for the future  pseudo $H$-type algebra and so we choose an even number $m=2k$ of generators $$e_1,\dots,e_k,e_{k+1},\dots, e_{2k}.$$ The pseudo $H$-type algebra is of step two, so we are interested in length two free nilpotent Lie algebra ${N}$.
Let us introduce a scalar product $(\cdot\,,\cdot)_{{N}}$ in ${N}$ of some signature, such that the basis vectors $e_1,\dots,e_{2k},e_1\times e_2,\dots, e_{2k-1}\times e_{2k}$
are orthogonal and non-null with respect to this product.  The scalar product $(\cdot\,,\cdot)_{N}$ splits as $(\cdot\,,\cdot)_{N}=(\cdot\,,\cdot)_{U}+(\cdot\,,\cdot)_{Z}$, where $Z=\spn\{e_1\times e_2,e_1\times e_3,\dots, e_{2k-1}\times e_{2k}\}$. 
This way the Lie algebra $N$ splits into 
\[
N=U\oplus_{\perp}Z=\spn\{e_1,\dots, e_{2k},\underbrace{e_1\times e_2,\dots, e_{2k-1}\times e_{2k}}_{k(2k-1)}\},
\]
Define the signature index of a scalar product $(\cdot\, ,\cdot)_{U}$ by
\[
\varepsilon_j(s,r)=
\begin{cases} 1, &  \text{for $j=1,\dots, s$;} \\
-1, & \text{for $j=s+1,\dots, s+r$.} 
\end{cases}
\]
Let us normalize the basis vectors  such that 
\begin{itemize}
\item $(e_j,e_j)_{U}=\varepsilon_j(s,r)/k$ for $j=1,\dots 2k$, $s+r=2k$, where $s,r$ are not specified so far; 
\item  $|(e_i\times e_j,e_i\times e_j)_{Z}|=1/k$.
\end{itemize}

If $A\subset Z$ is an ideal of $N$,  such that $N=U\oplus_{\perp}\Omega\oplus_{\perp}A$, and $\mathfrak{n}=N/A$, then $\mathfrak{h}:=U(\text{mod $A$})$, $\mathfrak{z}:=\Omega(\text{mod $A$})$, the scalar product $(\cdot,\cdot)_N=(\cdot,\cdot)_U+(\cdot,\cdot)_{\Omega}+(\cdot,\cdot)_A$ gives $(\cdot,\cdot)_{\mathfrak{n}}:=(\cdot,\cdot)_U+(\cdot,\cdot)_{\Omega}$. 
Let us denote by $\tilde{u}=u+A$, and $\tilde{v}=v+A$ for $u,v\in U$.
The commutator on $\mathfrak{n}$ is defined by $[\tilde{u},\tilde{v}]=u\times v+A$.

Let us define an operator $J_{\omega}\colon U\to U$, $\omega\in Z$, by $(J_{\omega}u,v)_{U}=(\omega,u\times v)_{Z}$. So far it is only a definition of $J$ without requiring the orthogonality condition \eqref{p2}.

In order to make $\mathfrak{n}$ a pseudo $H$-type algebra,  the operator $\tilde{J}_{z}$, $z\in \mathfrak{z}$, $z=\omega+A$, $\omega\in\Omega$, must satisfy
\[
(\tilde{J}_z\tilde{u},\tilde{v})_{\mathfrak{h}}=(z,[\tilde{u},\tilde{v}])_\mathfrak{z}=(\omega+A,u\times v+A))_{Z}=(\omega,u\times v)_{\Omega}=(J_{\omega}{u},{v})_{U}
\]
 and the orthogonality condition \eqref{p2}.

The next step is to construct the ideal  ${A}$ of ${N}$ as an orthogonal complement in $Z$ to $\Omega$. To define $\Omega$ we consider a family of partitions $\mathfrak{p}=\{\mathfrak{p}_l\}$, $l=1,\dots,\frac{(2k)!}{2^k k!}$,  of the set of integers $1,2,\dots, 2k$  in ordered  pairs $(i,j)$ , $i< j$.
For each $\mathfrak{p}_l\in \mathfrak{p}$ we construct 
ordered pairs $(e_i,e_j)$,  $(i,j)\in \mathfrak{p}_l$, the product $e_i\times e_j$,
and the  vector
\[
\omega_l=\sum\limits_{(i,j)\in \mathfrak{p}_l}\alpha^l_{ij}(e_i\times e_j)\in Z, \quad \alpha^l_{ij}\in \mathbb R.
\]
The  operator 
$J_{\omega_l}\colon {U}\to {U}$ is defined by the equality 
\begin{equation}\label{eq:J}
(J_{\omega_l}(u), v)_{U}=(\omega_l,u\times v)_{Z},\quad u,v\in U.
\end{equation}
Let us assume that the coefficients $\alpha^l_{ij}$ are chosen so that $Z=\spn\{\omega_l\colon l=1,\dots,\frac{(2k)!}{2^k k!}\}$.
From the definition of the operator $J_{\omega_l}$ and the product $(\times)$ it immediately follows that  $J_{\omega_l}$ extends to the operator $J_{\omega}\colon U\to U$,  $\omega\in Z$ by linearity, and
\[
(J_{\omega}(u), v)_U=-(u, J_{\omega}(v))_U.
\]

We restrict the number of partitions from $\mathfrak{p}$ assigning special values to the coefficients $\alpha^l_{ij}$ by requiring that the operators
$J_{\omega_l}$ act in a special way on the basis $\{e_1,\dots,e_{2k}\}$ of~$U$.
 
\begin{definition}
By {\it signed permutation} of vectors $e_1,\dots,e_{2k}$ from $U$ by the operator $J_{\omega}$ we understand a permutation  where some of resulting vectors may change orientation.
That is, for any $i=1,\dots, 2k$, there exists $j=1,\dots, 2k$, such that $J_{\omega}e_i=\pm e_j$.
\end{definition}

\begin{proposition}\label{prop3}
For a fixed $l$, the operator $J_{\omega_l}\colon U\to  U$ is a signed permutation of the basis $e_1,\dots,e_{2k}$, if and only if,  for every fixed $a\in\{1,\dots, 2k\}$ there exists a unique $b\in\{1,\dots, 2k\}$, $a\neq b$, such that 
$\alpha^l_{ab}=\pm 1$ if $a<b$ or $\alpha^l_{ba}=\pm 1$ if $a>b$. 
\end{proposition}
\begin{proof} The sufficient part is trivial. For the necessary part,
 fix some $a\in\{1,\dots, 2k\}$ and~$l$. Then in the partition $\mathfrak{p}_l$ there is unique index $b\in\{1,\dots, 2k\}$ such that $(a,b)\in\mathfrak{p}_l$. For simplicity assume that $a<b$. The definition~\eqref{eq:J} of the operator $J_{\omega_l}$ immediately implies
$$
(J_{\omega_l}(e_a), e_b)_{U}=(\omega_l,e_a\times e_b)_{Z}=\pm\frac{\alpha^l_{ab}}{k}\neq 0,
$$ 
and 
$$
(J_{\omega_l}(e_a), e_j)_{U}=(\omega_l,e_a\times e_j)_{Z}=0\quad\text{for all}\quad j\in\{1,\dots, 2k\},\ \ j\neq b.
$$
since $\omega_l$ does not contain the term $e_a\times e_j$. The map $J_{\omega_l}$ acts by sign permutation of the basis, therefore $J_{\omega_l}(e_a)=\pm e_b$ and this implies also that $\alpha^l_{ab}=\pm 1$. Here the sign $\pm$ means, + or -. 
\end{proof}

\begin{remark}
Observe that the map $J_{\omega_l}\colon U\to  U$ is injective in the case of signed permutation.
\end{remark}

The following simple statement can be found in standard texts in combinatorics and graph theory, e.g.~\cite{Anderson}.

\begin{proposition}[\cite{Anderson}] The set $\mathfrak{p}$ of all $\frac{(2k)!}{2^k k!}$ possible partitions contains a subset $$\mathfrak{p}'=\{\mathfrak{p}_{l_1},\dots, \mathfrak{p}_{l_{2k-1}}\},$$  such that each pair $(i,j)$, $1\leq i<j\leq 2k$, appears only once in these partitions. 
\end{proposition}

This proposition can be reformulated in terms of graph theory.
 Let us consider the complete graph $K_{2k}$ of $2k$ vertices at the numbers $1,\dots, 2k$, and recall that  a  1-factorization of a graph is a decomposition of all the edges of the graph into 1-factors, the sets of $k$ independent edges (without common vertices). The graph $K_{2k}$ has exactly $(2k-1)$ 1-factors, see e.g., \cite{Gross}, which in our case coincide with a possible set of partitions $\mathfrak{p}'=\{\mathfrak{p}_{l_1},\dots, \mathfrak{p}_{l_{2k-1}}\}$.

We observe that splitting of $K_{2k}$ into  $(2k-1)$ 1-factors is not unique and may be even not isomorphic, e.g., Kirkman~\cite{Kirkman} and Steiner~\cite{Steiner} tournaments
are not isomorphic~\cite{Anderson}, see example:
\[
\begin{array}{cccc} 12 & 38 & 47 & 56 \\ 
13 & 24 & 58 & 67 \\ 
14 & 26 & 35 & 78 \\
15 & 28 & 37  & 46 \\
16 & 23 & 57  & 48 \\
17 & 25  & 34 & 68 \\
18 & 27 & 36 & 45    
\end{array}
\qquad \mbox{and} \qquad
\begin{array}{cccc} 12 & 37 & 45 & 68 \\ 
13 & 27 & 48 & 56 \\ 
14 & 25 & 38 & 67 \\
15 & 24 & 36  & 78 \\
16 & 28 & 35  & 47 \\
17 & 23  & 46 & 58 \\
18 & 26 & 34 & 57    
\end{array}.
\]
Indeed, in the first array any two lines set together in one graph give a cycle of length 8, while in the second array this will give two disjoint cycles of length 4. 

Let $\omega_{l_m}$ be constructed by  $\mathfrak{p}_{l_m}\in\mathfrak{p}'$, and from now on, let the corresponding operator $J_{\omega_{l_m}}$ acts as a signed basis permutation on $U$. 
 Fixing some $e_j$ we obtain an operator $J(e_j)\colon Z\to U$. From now on we omit the upper index writing simply $\alpha_{ij}$ instead of $\alpha_{ij}^{l}$ because the pair $(i,j)$
 is met only once in every $\mathfrak{p}'$.

\begin{proposition}
Let   $\mathfrak{p}'=\{\mathfrak{p}_{l_1},\dots, \mathfrak{p}_{l_{2k-1}}\}$ be a set of partitions, where each pair $(i,j)$, $1\leq i<j\leq 2k$, appears only once.  
Let the operator $J_{\omega_{l_m}}$ acts over $U$ as a signed permutation.
The map $$J(e_j)\colon \spn\{\omega_{l_m}=\sum_{(i,j)\in \mathfrak{p}_{l_m}}\alpha_{ij}(e_i\times e_j), \,\,m=1,\dots, (2k-1)\}\to  \{e_j\}^{\perp}$$ is bijective.
\end{proposition}
\begin{proof} Without loss of generality, fix $j=1$. Then $J_{\omega_{l_m}}(e_1)= \pm e_{j_m}$ by Proposition~\ref{prop3}, where $m=1,\dots, (2k-1)$, the index $j_m\in\{2,3,\dots 2k\}$, and $j_{m}\neq j_n$ if $m\neq n$. If $\omega=\sum_{m=1}^{2k-1}c_m\omega_{l_m}$, then $J_{\omega}(e_1)=\sum_{m=1}^{2k-1}\pm c_me_{j_m}$, which vanishes if and only if, all coefficients $c_m=0$. Therefore, $J(e_1)$ acts as an isomorphism of linear spaces $\spn_{1\leq m\leq 2k-1}\{\omega_{l_m}\}$ and $\{e_1\}^{\perp}$. 
\end{proof}

The scalar product on $U$ is defined and now we fix it on $Z$  requiring $(\omega_{l_m}\,,\omega_{l_m})_{Z}=\pm 1$.
This can be always achieved by the following procedure. Fix $m$ and write
\[
(\omega_{l_m},\omega_{l_m})_Z=\sum\limits_{(i,j)\in\mathfrak{p}_{l_m}}(e_i\times e_j,e_i\times e_j)_Z.
\]
We choose a signature $(p,q)$ of $(\cdot,\,\cdot)_Z$, $p+q=2k-1$.  Then we define  $(e_i\times e_j, e_i\times e_j)_Z=1/k$ for $(i,j)\in\mathfrak{p}_{l_m}$, where $m=1,\dots, p$, and $(e_i\times e_j, e_i\times e_j)_Z=-1/k$ for  $m=p+1,\dots, p+q=2k-1$. We arrive at two options: the metric on $Z$ is positive definite, $q=0$, and
$(\omega_{l_m}\,,\omega_{l_m})_{Z}=1$, and the metric on $Z$ is indefinite non-degenerate, $q\neq 0$, so that $(\omega_{l_m}\,,\omega_{l_m})_{Z}=-1$ for some of ${l_m}$.



\subsection{Inner product on $Z$}

Since we have chosen an inner (positive definite) product  on~$Z$, the scalar product on $U$ may by either an inner product or a scalar negative definite product, because of the
properties of the operator $J$. Both pictures are isomorphic, hence we choose an inner product on the whole $N$.
In this section we study the case of both metrics $(\cdot\,,\cdot)_{U}$ and $(\cdot\,,\cdot)_{Z}$ to be positive definite. In this case
\begin{equation}\label{otno}
J_{\omega_l}^2=-\id_{U} \quad\text{and} \quad J_{\omega_{l_1}}J_{\omega_{l_2}}=-J_{\omega_{l_2}}J_{\omega_{l_1}}.
\end{equation}

For every $\mathfrak{p}_{l}\in\mathfrak{p}'$, for the corresponding vector $\omega_l$, and finally,  for the operator $J_{\omega_l}$ acting by signed permutation, we associate a $2k\times 2k$ matrix of coefficients
$E_l=\{\alpha_{ij}\}$ where $\alpha_{ij}=-\alpha_{ji}$ and $\alpha_{ij}=0$ if $(i,j)\not\in \mathfrak{p}_l$. Matrices $E_l$, $l=1,\dots, 2k-1$ are orthogonal with entries $0,\pm 1$.
 So we  constructed an injective homomorphism of the operators $J_{\omega_l}$ to the set of orthogonal matrices.

We continue with the Hurwitz-Radon-Eckmann theorem, see \cite{Eckmann, Kaplan}. The Hurwitz-Radon function is defined as  $\rho(n):=8\alpha+2^\beta$, where $n$ is uniquely represented by $n=u2^{4\alpha+\beta}$  with $u$  odd, $\beta=0,1,2$ or $3$.

\begin{theorem}[\cite{Eckmann, Hurwitz, Radon}]\label{Eckmann}
A family of $2k-1$ real orthogonal $2k\times 2k$ matrices $E_1,\dots,E_{2k-1}$ admits at most $\rho(2k)-1$ matrices satisfying $E_j^2=-I$ and $E_iE_j=-E_jE_i$, where $\rho(\cdot)$ is the Hurwitz-Radon function.
\end{theorem}

The function $\rho(n)$ seems to be quite irregular from the first glance but it is not so. It is periodic. We did not find a good reference so let us prove this simple but useful fact.

\begin{proposition}\label{periodic} If $n= u 2^r$, $r\in \mathbb N$, then $\rho(n+2^R)=\rho(n)$ for any $R\neq r$.
\end{proposition}
\begin{proof}
If $r=0$, then $n$ is odd and $\rho(n)=1$. So the conclusion is trivial.
Let  $r =4\alpha_1+\beta_1>0$, $n=u_1 2^{4\alpha_1+\beta_1}$,  and let $n+2^R$ is represented as $n+2^R=u2^{4\alpha_2+\beta_2}$, where $\alpha_k, \beta_k$, $k=1,2$ are as above, and
both $u_1$ and $u_2$ are odd. Without loss of generality assume that $4\alpha_1+\beta_1\leq 4\alpha_2+\beta_2$. Then,
\[
2^{4\alpha_1+\beta_1}\left(u_22^{4(\alpha_2-\alpha_1)+(\beta_2-\beta_1)}-u_1\right)=2^{r}\left(u_22^{4(\alpha_2-\alpha_1)+(\beta_2-\beta_1)}-u_1\right)=2^R.
\]
Since $R\neq r$, the expression  $u_22^{4(\alpha_2-\alpha_1)+(\beta_2-\beta_1)}-u_1$ is even which is possible only if $4(\alpha_2-\alpha_1)+(\beta_2-\beta_1)=0$. Since $|\beta_2-\beta_1|<4$, we conclude that $\alpha_1=\alpha_2$, and then, $\beta_1=\beta_2$. The definition of the Huwitz-Radon function implies that $\rho(n+2^R)=\rho(n)$.
\end{proof}

\begin{remark} If $R=r$, then the periodicity does not hold in general, observe e.g., $\rho(12)=\rho(3\cdot 2^2)=4\neq 9=\rho(16)=\rho(3\cdot 2^2+2^2)$.
\end{remark}

The sequence of values $\{\rho(n)\}_{n\geq 1}$ satisfies the following telescopic property, see Figure~1.

\begin{corollary}
The sequence $\{\rho(n)\}_{n=1}^{2^r-1}$ is the same as $\{\rho(n)\}_{n=2^r+1}^{2^{r+1}-1}$.
\end{corollary}
\begin{proof}
Indeed, the number of elements in both sequences is $2^r-2$, all values $n=1,\dots, 2^r-1<2^r$, therefore, \[\{\rho(n)\}_{n=1}^{2^r-1}=\{\rho(n+2^r)\}_{n=1}^{2^r-1}=\{\rho(n)\}_{n=2^r+1}^{2^{r+1}-1}\]
by the previous proposition.
\end{proof}


Theorem~\ref{Eckmann} was proved for integer matrices $E_k$ having entries among $\{0,\pm 1\}$, i.e., for a special kind of weighing matrices in~\cite{GP, Geramita79}. Moreover, the upper bound $\rho(2k)-1$ is achieved for integer Hurwitz-Radon matrices \cite[Theorem 1.6]{Geramita79}.
Relations~\eqref{otno} and Theorem~\ref{Eckmann} reveal the fact that  the HR-family of orthogonal matrices gives a representation of a Clifford algebra  $Cl_{p,0}$.

\begin{figure}
\begin{pspicture}(0,0)(18,5)
\psline[linestyle=dashed](0,2.5)(16,2.5)
\psline[linestyle=dashed](1.3,1.5)(1.3,3.5)
\rput(0.5,3){$n$} \rput(1.7,3){$1$}
\rput(0.5,2){$\rho(n)$} \rput(1.7,2){$1$}
\psframe(2,1.5)(2.7,3.5) \rput(2.35,2){$2$}\rput(2.35,3){$2$}\rput(2.4,4){$2^1$}
\rput(3,2){$1$}\rput(3,3){$3$}
\psframe(3.3,1.5)(4,3.5) \rput(3.65,2){$4$}\rput(3.65,3){$4$}\rput(3.7,4){$2^2$}
\rput(4.3,2){$1$}\rput(4.3,3){$5$}
\rput(4.8,2){$2$}\rput(4.8,3){$6$}
\rput(5.3,2){$1$}\rput(5.3,3){$7$}
\psframe(5.6,1.5)(6.3,3.5) \rput(5.95,2){$8$}\rput(5.95,3){$8$}\rput(6,4){$2^3$}
\rput(6.6,2){$1$}\rput(6.6,3){$9$}
\rput(7.1,2){$2$}\rput(7.1,3){$10$}
\rput(7.7,2){$1$}\rput(7.7,3){$11$}
\rput(8.3,2){$4$}\rput(8.3,3){$12$}
\rput(8.9,2){$1$}\rput(8.9,3){$13$}
\rput(9.5,2){$2$}\rput(9.5,3){$14$}
\rput(10.1,2){$1$}\rput(10.1,3){$15$}
\psframe(10.4,1.5)(11.1,3.5) \rput(10.75,2){$9$}\rput(10.75,3){$16$}\rput(10.75,4){$2^4$}
\rput(11.4,2){$1$}\rput(11.4,3){$17$}
\rput(12,2){$2$}\rput(12,3){$18$}
\rput(12.6,2){$1$}\rput(12.6,3){$19$}
\rput(13.2,2){$4$}\rput(13.2,3){$20$}
\rput(13.8,2){$1$}\rput(13.8,3){$21$}
\rput(14.4,2){$2$}\rput(14.4,3){$22$}
\rput(15,2){$1$}\rput(15,3){$23$}
\rput(15.6,2){$\dots$}\rput(15.6,3){$\dots$}
\psline(1.7,1.6)(1.7,1.2)
\psline(1.7,1.2)(3,1.2)
\psline{<-}(3,1.6)(3,1.2)
\psline(4.3,1.6)(4.3,1.2)
\psline(4.3,1.2)(5.3,1.2)
\psline(5.3,1.6)(5.3,1.2)
\psline(2.35,1.2)(2.35,1)
\psline(2.35,1)(4.8,1)
\psline{<-}(4.8,1.2)(4.8,1)
\psline(6.6,1.6)(6.6,1.2)
\psline(6.6,1.2)(10.1,1.2)
\psline(10.1,1.6)(10.1,1.2)
\psline(3.3,1)(3.3,0.8)
\psline(3.3,0.8)(8.3,0.8)
\psline{->}(8.3,0.8)(8.3,1.2)
\psline(11.4,1.6)(11.4,1.2)
\psline(11.4,1.2)(15,1.2)
\psline[linestyle=dashed](15,1.2)(16,1.2)
\psline(5.95,0.8)(5.95,0.5)
\psline(5.95,0.5)(14.4,0.5)
\psline{->}(14.4,0.5)(14.4,1.2)
\end{pspicture}
\caption{Telescopic property}
\end{figure}
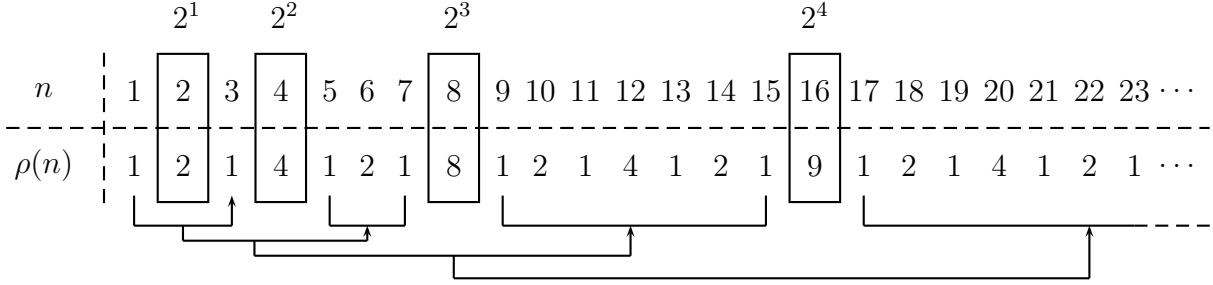

\begin{theorem}\label{t6}
Given a free two-step nilpotent  Lie algebra $N=U\oplus_{\perp}Z$ with an inner product and dim$(U)=2k$, there exists an ideal $A$ such that $N/A$ is an $H$-type algebra $\mathfrak{n}=\mathfrak{h}\oplus_{\perp} \mathfrak{z}$, where $\mathfrak{h}=U$ and the center is $\mathfrak{z}=Z/A$. Moreover, the ideal $A$ can be chosen such that $\dim(\mathfrak{z})$ takes every value from $\{1,\dots,\rho(2k)-1\}$. 
\end{theorem}
\begin{proof} We reformulate this statement in the language of partitions as follows.
There exists a family of partitions $$\mathfrak{p}''=\{\mathfrak{p}''_{1},\dots, \mathfrak{p}''_{\rho(2k)-1}\}\subset \mathfrak{p}',$$ equivalent to a subset of 1-factorization of
the complete graph $K_{2k}$, such that the corresponding set of vectors $\{\omega_1,\dots,\omega_{\rho(2k)-1}\}$ with the coefficients $\alpha_{ij}=\pm 1$ defines $A$ as the orthogonal
complement to
span$\{\omega_1,\dots,\omega_{\rho(2k)-1}\}$ in $Z$, and $\mathfrak{z}=Z/A$ is the center of an $H$-type algebra. In what follows we describe the choice of $\mathfrak{p}''$.

We start by analysing the correspondence between the operators $J_{\omega_l}$, $l\in 1,\dots, 2k-1$, and the orthogonal matrices $E_l$. Indeed, an orthogonal matrix $E_l$ with entries $0,\pm 1$ is a permutation of the standard basis of $\mathbb R^{2k}$, if an only if, each row and each column contains only one entry $\pm1$ and others are 0. Respectively, $E_l$ is  a signed  permutation of the canonical basis, if and only if,
each row and each column contains only one entry 1 or (-1) and others are 0. The form of the partitions $\mathfrak{p}'$ guarantees that the matrices $E_l$ constructed
above with the entries $\alpha_{ij}$, $(i,j)\in \mathfrak{p_l}\in\mathfrak{p}'$ for $i<j$ and $\alpha_{ij}=-\alpha_{ji}$, are the sign permutation of the basis $\{e_1,\dots, e_{2k}\}$, i.e., we constructed a one-to-one correspondence
between the orthogonal sign permutation matrices $E_l$, satisfying the conditions of Theorem~\ref{Eckmann}, and the operators $J_{\omega_l}$. 

Given the set of integer matrices $\{E_{l_1},\dots E_{l_{\rho(2k)-1}}\}$, representing the maximal family of matrices in Theorem~\ref{Eckmann}, we set  
$\mathfrak{p}''$ to be the family of partitions
$\mathfrak{p}_{l_k}$ consisting of elements $(i,j)$ such that $\alpha_{ij}$ is the entry in the upper triangular part of the matrix $E_{l_k}$. According to $\mathfrak{p}''$,
we construct the vectors $\omega_{l_1},\dots, \omega_{l_{\rho(2k)-1}}$. Corresponding operators $J_{\omega_l}$ satisfy relations \eqref{otno} since the matrices $E_{l}$ do.
Now we set $\mathfrak{h}=U$ and $A$ to be the orthogonal complement in $Z$ to $\spn\{\omega_{l_1},\dots, \omega_{l_{\rho(2k)-1}}\}$. Then we define $\mathfrak{z}=Z/A$, and $\mathfrak{h}\oplus_{\perp} \mathfrak{z}$ is an $H$-type algebra $\mathfrak{n}$ by Proposition~\ref{p2}.
\end{proof}

The following theorem was proved in~\cite{Eber04, FurutaniMarkina, CrDod} by different methods.

\begin{theorem}\label{th.rational}
The $H$-type algebras admit rational structure constants.
\end{theorem}
\begin{proof}
Let $\mathfrak{h}=U$, and let $\mathfrak{z}=Z/A$ be chosen and  spanned by the vectors $\omega_l=\sum_{(i,j)\in \mathfrak{p}''_l}\alpha_{ij}\,e_i\times e_j$ modulo the ideal $A$.
Then $[e_n,e_m]=e_n\times e_m\,\text{(mod A)}\in \mathfrak{z}$. Let $(n,m)\in \mathfrak{p}''_l$ for some $l$. Write
\[
\omega_l=\alpha_{nm}e_n\times e_m+\sum_{(i,j)\in (\mathfrak{p}''_l\setminus (n,m))}\, \alpha_{ij}\,e_i\times e_j,
\]
and set
\[
\omega_{(l;ij)}=\alpha_{nm}e_n\times e_m-\alpha_{ij}\,e_i\times e_j,\quad (i,j)\in (\mathfrak{p}''_l\setminus (n,m))
\]
Then, $\omega_l\perp \omega_{(l;ij)}$ and $\omega_{(l;ij)}\in A$.
So we write $$e_n\times e_m\,\text{(mod A)}=\frac{1}{k}\left(\alpha_{nm}\omega_l+\sum_{(i,j)\in (\mathfrak{p}''_l\setminus (n,m))}\omega_{(l;ij)}\right)\, \text{(mod A)}=\frac{\alpha_{nm}}{k}\omega_l \,\text{(mod A)}.$$
\end{proof}

\begin{remark} Observe that $\omega_l= \frac{1}{2}\sum_{j=1}^{2k}e_j\times E_le_j =\frac{1}{2}\sum_{j=1}^{2k}e_j\times J_{\omega_l}e_j$.
\end{remark}

As it was observed the existence of integer matrices $\{E_{l_1},\dots E_{l_{\rho(2k)-1}}\}$ from the Hurwitz-Radon family was proved in \cite[Theorem 1.6]{Geramita79}.
The orthogonality condition assures that they are skew-symmetric and each of raws and columns contains exactly one non-zero element. 

We are not focused on algorithms of calculation of matrices $E_l$, vectors $\omega_l$ and operators $J_{l}$ in full generality which is the subject of computational algebra, however, we look at an illustrative example in the next subsection.

\subsubsection{The case $k=2$} Let us have a close look at the simplest example of $U$ generated by $\{e_1,e_2,e_3,e_4\}$ and
\[
\begin{array}{ccc}
\omega_1 &=& \alpha_{12}e_1\times e_2+\alpha_{34}e_3\times e_4,\\
\omega_2 &=& \alpha_{13}e_1\times e_3+\alpha_{24}e_2\times e_4,\\
\omega_3 &= &\alpha_{14}e_1\times e_4+\alpha_{23}e_2\times e_3,
\end{array}
\]
where $\omega_i$ are orthogonal with respect to the inner product. That is we have chosen $\mathfrak{p}''=\mathfrak{p}'\subset \mathfrak{p}$ such that $\mathfrak{p}''=\{\mathfrak{p}_1,\mathfrak{p}_2,\mathfrak{p}_3\}$, where
\[
\begin{array}{cccc} \mathfrak{p}_1&=&12 & 34,  \\ 
 \mathfrak{p}_2&=&13 & 24,  \\
  \mathfrak{p}_3&=&14 & 23.   
\end{array}
\]
Definition of the operator $J$ and the choice of the inner product imply that $J^2_{\omega_m}=-\id_{U}$. In order to define the coefficients $\alpha$ we need to check anti-commutativity of 
$J_{\omega_1}$, $J_{\omega_2}$, and $J_{\omega_3}$ on the basis. We calculate
\[
\begin{array}{lcccr}
e_1&\stackrel{J_{\omega_1}}\longrightarrow & \alpha_{12}e_2 &\stackrel{J_{\omega_2}}\longrightarrow&\,\,\,\, \alpha_{12}\alpha_{24}e_4,\\
e_1&\stackrel{J_{\omega_2}}\longrightarrow & \alpha_{13}e_3 &\stackrel{J_{\omega_1}}\longrightarrow&\,\, \,\,\alpha_{13}\alpha_{34}e_4,
\end{array}
\]
\[
\begin{array}{lcccr}
e_1&\stackrel{J_{\omega_2}}\longrightarrow & \alpha_{13}e_3 &\stackrel{J_{\omega_3}}\longrightarrow& -\alpha_{13}\alpha_{23}e_2,\\
e_1&\stackrel{J_{\omega_3}}\longrightarrow & \alpha_{14}e_4 &\stackrel{J_{\omega_2}}\longrightarrow& -\alpha_{14}\alpha_{24}e_2,
\end{array}
\]
\[
\begin{array}{lcccr}
e_1&\stackrel{J_{\omega_3}}\longrightarrow & \alpha_{14}e_4 &\stackrel{J_{\omega_1}}\longrightarrow& -\alpha_{14}\alpha_{34}e_3,\\
e_1&\stackrel{J_{\omega_1}}\longrightarrow & \alpha_{12}e_2 &\stackrel{J_{\omega_3}}\longrightarrow& \alpha_{12}\alpha_{23}e_3,
\end{array}
\]
and in order to satisfy anticommutativity we write the homogeneous equations
\begin{equation}\label{system}
\alpha_{12}\alpha_{24}+\alpha_{13}\alpha_{34}=0, \quad \alpha_{13}\alpha_{23}+
\alpha_{14}\alpha_{24}=0, \quad \alpha_{14}\alpha_{34}-\alpha_{12}\alpha_{23}=0.
\end{equation}
Analogously we start with the vectors $e_2,e_3,$ and $e_4$. We arrive at three additional equations
\begin{equation}\label{system2}
\alpha_{12}\alpha_{13}+\alpha_{24}\alpha_{34}=0, \quad \alpha_{13}\alpha_{14}+
\alpha_{23}\alpha_{24}=0, \quad \alpha_{12}\alpha_{14}-\alpha_{23}\alpha_{34}=0,
\end{equation}
which we solve over the integers $(1,-1)\in \mathbb Z$. However,  the systems \eqref{system} and \eqref{system2} are equivalent. Indeed, multiplying the first equation in \eqref{system2} by
$\alpha_{12}\alpha_{34}$ we obtain the first equation in \eqref{system}, multiplying the second equation in \eqref{system2} by
$\alpha_{14}\alpha_{23}$ we obtain the second equation in \eqref{system}, and finally multiplying the third equation in \eqref{system2} by
$\alpha_{13}\alpha_{24}$ we obtain the third equation in \eqref{system}. This  reflects the fact that we can  check anticommutativity of $J_{\omega_1}$, $J_{\omega_2}$, and $J_{\omega_3}$ starting only with one vector, e.g., $e_1$. 

The operators $J_{\omega_1}$, $J_{\omega_2}$, and $J_{\omega_3}$ define the orthogonal matrices $E_1$, $E_2$, and $E_3$ as it was described in Theorem~\ref{t6}
\[
E_1=\left(
\begin{array}{rrrr} 0 & -\alpha_{12} & 0 & 0 \\ 
\alpha_{12} & 0 & 0 & 0 \\ 
0 & 0 & 0 & -\alpha_{34} \\
0 & 0 &  \alpha_{34}  & 0 \\  
\end{array}
\right),\quad E_2=\left(
\begin{array}{rrrr} 0 & 0 & -\alpha_{13} & 0 \\ 
0 & 0 & 0 & -\alpha_{24} \\ 
\alpha_{13} & 0 & 0 & 0 \\
0 & \alpha_{24} &  0  & 0 \\  
\end{array}
\right),
\]
\[
E_3=\left(
\begin{array}{rrrr} 0 & 0 & 0 & -\alpha_{14} \\ 
0 & 0 & -\alpha_{23} & 0 \\ 
0 & \alpha_{23} & 0 & 0 \\
\alpha_{14} & 0 &  0  & 0 \\  
\end{array}
\right),
\]
which satisfy the conditions 
\begin{equation}\label{CondE}
E_j^2=-I\quad \text{and}  \quad E_iE_j=-E_jE_i.
\end{equation}
 In this case, in view of Hurwitz-Radon-Eckmann's Theorem~\ref{Eckmann}, $u=1$, $\alpha=0$, $\beta=2$, and $\rho(4)-1=3$.

The system (\ref{system}) has a non-unique solution in integers $\pm 1$. This fact can be viewed in the context of algebraic geometry and computational algebra. At the same time, this fact
refers to the existence of a special type of weighing matrices, see \cite{Geramita79} in which each line of $E_j$ contains only one non-zero entry. Let us define a symmetric square $\mathbb{S}^2(\mathfrak{h})=\mathfrak{h}\vee \mathfrak{h}$ of the vector space $\mathfrak{h}=\spn\{(x_1,\dots, x_m)^t\}$ of dimension $m=2k$, and let $\mathbb{P}(\mathbb{S}^2(\mathfrak{h}))$ be its projectivization of (projective) dimension $\frac{1}{2}(m-1)(m+2)$,
\[
\mathbb{P}(\mathbb{S}^2(\mathfrak{h}))=\{(y_{11},y_{12},\dots,y_{ij},\dots, y_{mm}),\quad y_{ij}\neq 0, \quad i\leq j\}.
\]
The Veronese embedding $\varphi \colon \mathbb{P}(\mathfrak{h})\to \mathbb{P}(\mathbb{S}^2(\mathfrak{h}))$ is defined by
\[
(x_1,\dots, x_m)^t\mapsto (x_1^2,x_1x_2,\dots,x_1x_m,x_2^2,x_2x_3\dots, x_m^2)^t.
\]
The Veronese variety $\mathcal{V}(\mathfrak{h})$, see e.g., \cite{Cox}, is the image $\varphi( \mathbb{P}(\mathfrak{h}))=\mathcal{V}(\mathfrak{h})\subset \mathbb{P}(\mathbb{S}^2(\mathfrak{h}))$. The intersection of the Veronese variety  $\mathcal{V}(\mathfrak{h})$ with a generic $(\frac{1}{2}(m-1)(m+2)-m+1)$-dimensional hyperplane has $2^{m-1}$ points, so $\mathcal{V}(\mathfrak{h})$ is of order $2^{m-1}$.

The vector space $\mathbb{S}^2(\mathfrak{h})$ and its projectivization $\mathbb{P}(\mathbb{S}^2(\mathfrak{h}))$ can also be realized in terms of symmetric $m\times m$-matrices
\[
\mathbf{y}=\left(
\begin{array}{rrrr} y_{11} & y_{12} & \dots & y_{1m} \\ 
y_{12} & y_{22} &\dots & y_{2m} \\ 
\dots & \dots & \dots & \dots \\
y_{1m} & y_{2m} & \dots  & y_{mm} \\  
\end{array}
\right).
\]
Then the Veronese embedding $\varphi \colon \mathbb{P}(\mathfrak{h})\to \mathbb{P}(\mathbb{S}^2(\mathfrak{h}))$ is defined as
\[
\left(
\begin{array}{c} x_{1} \\
x_2\\
\vdots\\
x_m
\end{array}
\right)\mapsto
\left(
\begin{array}{c} x_{1} \\
x_2\\
\vdots\\
x_m
\end{array}
\right)(x_1,x_2,\dots,x_m)=
\left(
\begin{array}{cccc} x_{1}^2 & x_{1}x_2 & \dots & x_{1}x_m \\ 
x_1 x_2 & x_2^2 &\dots & x_{2} x_m \\ 
\dots & \dots & \dots & \dots \\
x_{1}x_m & x_{2}x_m & \dots  & x_{m}^2 \\  
\end{array}
\right).
\]
A point $\mathbf{y}$ is in $\mathcal{V}(\mathfrak{h})$ if and only if the matrix corresponding to $\mathbf{y}$ is of rank 1. An $m\times m$ symmetric matrix has
$\frac{1}{12}(m-1)m^2(m+1)$ independent 2-minors, which all vanish for the rank 1 matrix $\mathbf{y}$. Each 2-minor defines a quadric in $\mathbb{P}(\mathbb{S}^2(\mathfrak{h}))$
and $\mathbf{y}\in \mathcal{V}(\mathfrak{h})$ if and only if $\mathbf{y}$ lies in the intersection of these quadrics. The linearization algorithm can be viewed as follows, see e.g., \cite{Murphy}.
Consider a system of $p$ independent homogeneous quadratic equations $f_1=0,\dots,$ $f_p=0$ with the projective point $\mathbf{x}$ as a non-trivial solution. This system
is projected to the linear system of $p$ equations $\hat{f}_1=0,\dots, \hat{f}_p=0$ in terms of variables $\mathbf{y}$. Define the projective subspace $\mathcal H\subset \mathbb{P}(\mathbb{S}^2(\mathfrak{h}))$ as the intersection of $p$ hyperplanes defined by $\hat{f}_1=0,\dots, \hat{f}_p=0$. So we have that the solution $\mathbf{x}$ is embedded into
the intersection $\mathcal H\cap \mathcal{V}(\mathfrak{h})$. Thus the solution is obtained by $\mathbf{x}=\varphi^{-1}(\mathcal H\cap \mathcal{V}(\mathfrak{h}))$. Since $\mathcal{V}(\mathfrak{h})$ contains no non-trivial linear subspaces,  the algorithm is efficient if $\mathcal H$ consists of a finite number of isolated points.

Let us denote $x_1=\alpha_{12}$, $x_2=\alpha_{24}$, $x_3=\alpha_{13}$, $x_4=\alpha_{34}$, $x_5=\alpha_{23}$, $x_6=\alpha_{14}$, then
$y_{12}=\alpha_{12}\alpha_{24}$, $y_{13}=\alpha_{12}\alpha_{13}$, $y_{24}=\alpha_{24}\alpha_{34}$, $y_{34}=\alpha_{13}\alpha_{34}$, $y_{35}= \alpha_{13}\alpha_{23}$, $y_{26}=\alpha_{14}\alpha_{24}$, $y_{46}=\alpha_{14}\alpha_{34}$, $y_{15}=\alpha_{12}\alpha_{23}$, $y_{36}=\alpha_{13}\alpha_{14}$, $y_{25}=\alpha_{23}\alpha_{24}$, $y_{16}=\alpha_{12}\alpha_{14}$, $y_{45}=\alpha_{23}\alpha_{34}$. 
The equations \eqref{CondE} are equivalent to the linear homogeneous system $\hat{f}_1=0,\dots, \hat{f}_{8}=0$
\[
y_{11}-y_{22}=0,\quad y_{11}-y_{33}=0, \quad y_{11}-y_{44}=0,\quad y_{11}-y_{55}=0,\quad y_{11}-y_{66}=0,
\]
\[
y_{12}+y_{34}=0,\quad y_{35}+y_{26}=0,\quad y_{46}-y_{15}=0.
\]
In the projective coordinates this system has 3 solutions which we write in the matrix form as
\[
\begin{pmatrix}
1 & 1 & y_{13} &  y_{14} & 1& y_{16} \\
1 & 1 & y_{23} &  y_{24} & y_{25} & -1 \\
y_{13} & y_{23} & 1 &  -1 &1 & y_{36} \\
y_{14} & y_{24} & -1& 1& y_{45} & 1 \\
1 & y_{25} & 1&  y_{45} & 1 & y_{56} \\
y_{16} & -1 & y_{36} &  1 & y_{56} & 1 
\end{pmatrix},\quad
\begin{pmatrix}
1 & 1 & y_{13} &  y_{14} & 1& y_{16} \\
1 & 1 & y_{23} &  y_{24} & y_{25} & 1 \\
y_{13} & y_{23} & 1 &  -1 & -1 & y_{36} \\
y_{14} & y_{24} & -1 & 1& y_{45} & 1 \\
1& y_{25} & -1 &  y_{45} & 1 & y_{56} \\
y_{16} & 1 & y_{36} &  1 & y_{56} & 1 
\end{pmatrix},
\]
\[
\begin{pmatrix}
1 & 1 & y_{13} &  y_{14} & -1& y_{16} \\
1 & 1 & y_{23} &  y_{24} & y_{25} & -1 \\
y_{13} & y_{23} & 1 &  -1 & 1 & y_{36} \\
y_{14} & y_{24} & -1 & 1& y_{45} & -1 \\
-1 & y_{25} &1 &  y_{45} & 1 & y_{56} \\
y_{16} & -1 & y_{36} &  -1 & y_{56} & 1 
\end{pmatrix}.
\]
In the original variables we have 6 solutions in coordinates 
\[
(\alpha_{12},\alpha_{24},\alpha_{13},\alpha_{34}, \alpha_{23},\alpha_{14})
\]
as
\[
(1,1,1,-1,1,-1),\quad (1,1,-1,1,1,1),\quad (1,1,-1,1,-1,-1),
\]
\[
(-1,-1,-1,1,-1,1),\quad (-1,-1,1,-1,-1,-1),\quad (-1,-1,1,-1,1,1).
\]
The values of the remaining `free' variables $y_{13},\dots,  y_{56}$ then follow. 
Fixing $\alpha_{12}=1$ (or the first line of solutions) we arrive at three groups of vectors
\begin{equation}\label{s1}
\Omega_1=\spn\left(\begin{array}{ccc}
\omega^{\Omega_1}_1 &=& e_1\times e_2-e_3\times e_4\\
\omega^{\Omega_1}_2 &=& e_1\times e_3+e_2\times e_4\\
\omega^{\Omega_1}_3 &= & -e_1\times e_4+e_2\times e_3
\end{array}\right),
\end{equation}
\begin{equation}\label{s2}
\Omega_2=\spn\left(
\begin{array}{ccc}
\omega^{\Omega_2}_1 &=& e_1\times e_2+e_3\times e_4\\
\omega^{\Omega_2}_2 &=& -e_1\times e_3+e_2\times e_4\\
\omega^{\Omega_2}_3 &= & e_1\times e_4+e_2\times e_3
\end{array}\right),
\end{equation}
and
\begin{equation}\label{s3}
\Omega_3=\spn\left(
\begin{array}{ccc}
\omega^{\Omega_3}_1 &=& e_1\times e_2+e_3\times e_4\\
\omega^{\Omega_3}_2 &=& -e_1\times e_3+e_2\times e_4\\
\omega^{\Omega_3}_3 &= & -e_1\times e_4-e_2\times e_3
\end{array}\right).
\end{equation}
Recall that now we construct the ideals $A_1$, $A_2$, and $A_3$ as the orthogonal complements to $\Omega_1$, $\Omega_2$, and $\Omega_3$ in $Z$:
\begin{equation}\label{a1}
A_1=\spn\left(\begin{array}{ccc}
\omega^{A_1}_1 &=& e_1\times e_2+e_3\times e_4\\
\omega^{A_1}_2 &=& e_1\times e_3-e_2\times e_4\\
\omega^{A_1}_3 &= & -e_1\times e_4-e_2\times e_3
\end{array}\right),
\end{equation}
\begin{equation}\label{a2}
A_2=\spn\left(
\begin{array}{ccc}
\omega^{A_2}_1 &=& e_1\times e_2-e_3\times e_4\\
\omega^{A_2}_2 &=& -e_1\times e_3-e_2\times e_4\\
\omega^{A_2}_3 &= & e_1\times e_4-e_2\times e_3
\end{array}\right),
\end{equation}
and
\begin{equation}\label{a3}
A_3=\spn\left(
\begin{array}{ccc}
\omega^{A_3}_1 &=& e_1\times e_2-e_3\times e_4\\
\omega^{A_3}_2 &=& -e_1\times e_3-e_2\times e_4\\
\omega^{A_3}_3 &= & -e_1\times e_4+e_2\times e_3
\end{array}\right).
\end{equation}
Then, $\mathfrak{z}_j=\Omega_j(\text{mod $A_j$})=Z/A_j$, $j=1,2,3$.
The corresponding $H$-type groups are $\mathfrak{n}_j=\mathfrak{h}_j\oplus_{\perp}\mathfrak{z}_j$, where $\mathfrak{h}_j=U(\text{mod $A_j$})$. Fixing $\alpha_{12}=-1$ leads us to the inverse  matrices $E^{-1}_{l}=-E_{l}$ and the inverse operators $J^{-1}_{\omega_l}=-J_{\omega_l}$ corresponding to $\omega_l$, which are obtained also from the solutions to~\eqref{system}.

\subsection{Lie algebra isomorphism of solutions}

Let us remind some facts from the general theory of Lie algebras, see, e.g., \cite[Page 30]{Knapp}. Let $\mathfrak{g}$ be a Lie algebra, and 
let $\mathfrak{a}$  be an  ideal in  $\mathfrak{g}$. Then the vector space quotient $\mathfrak{g}/\mathfrak{a}$ becomes a Lie algebra under the definition of Lie brackets $[u+\mathfrak{a}, v+\mathfrak{a}]=[u,v]+\mathfrak{a}$. The quotient map $q\colon \mathfrak{g}\to\mathfrak{g}/\mathfrak{a}$ is a Lie algebra homomorphism and $\mathfrak{a}=\ker q$. Let $\pi\colon \mathfrak{g}\to \mathfrak{h}$ be another Lie algebra homomorphism. If $\mathfrak{a}=\ker\pi$, then $p\colon\mathfrak{g}/\mathfrak{a}\to\mathfrak{h}$, $\pi=p\circ q$, is a Lie algebra isomorphism.
 
In our case,  $N=U\oplus \Omega_j\oplus A_j$ and $\mathfrak{n}_j=N/A_j$. Let $B\colon N\to N$ be a free Lie algebra isomorphism such that $B(A_i)=A_j$. Then, denote
$q_j\colon N\to N/A_j$, $\pi= q_i\circ B$, and $q=q_j$. Then there is a Lie algebra isomorphism $p\colon N/A_j\to N/A_i$.

Let us now check the isomorphism of $H$-type structures on $\mathfrak{n}_j$. 
We construct the operator $B$ on $N$ such that
\begin{enumerate}
\item $B$ is a  free Lie algebra $N$ isomorphism;
\item $B(A_i)=A_j$;
\item $B$ is an isometry of $N$ preserving the orthogonal decomposition $$B\colon U\oplus_{\bot} \Omega_i\oplus_{\bot} A_i\to U\oplus_{\bot} \Omega_j\oplus_{\bot} A_j.$$
\end{enumerate}

The condition (1) implies that $B(u\times v)=Bu\times Bv$ for all $u,v\in U$, and (3) implies $(Bu, Bv)_{U}=(u,v)_U$ and
\begin{equation}\label{om1}
(B\omega, Bu\times Bv)_Z=(\omega, u\times v)_Z,\quad \text{for all $\omega\in Z$}.
\end{equation}
The condition (3) implies also that $B^{T}=B^{-1}$. Then
\begin{equation}\label{om2}
(J_{B\omega}Bu,Bv)_U=(J_{\omega}u,v)_U, \quad \text{for all $\omega\in Z$},
\end{equation}
and therefore,
 \begin{equation}\label{om3}
J_{B\omega}=BJ_{\omega}B^{-1}.
\end{equation}
 Obviously, the operators $J_{\omega}$ and $BJ_{\omega}B^{-1}$ satisfy conditions~\eqref{otno} on $\Omega_i$ and $\Omega_j$ respectively.
 The conditions (1-2) imply the Lie algebra isomorphism $p\colon \mathfrak{n_i}\to\mathfrak n_j$ induced by $B$. The quotient maps $q_j\colon N\to N/A_j=\mathfrak n_j$ are the orthogonal projections, hence $p$ is an isometry preserving the orthogonal decompositions $p\colon \mathfrak{h}_i\oplus_{\bot}\mathfrak{z}_i \to  \mathfrak{h}_j\oplus_{\bot}\mathfrak{z}_j$. This will guarantee that
 \[
 (z, [\tilde{u},\tilde{v}])_{\mathfrak{z}_i}= (pz, [p\tilde{u},p\tilde{v}])_{\mathfrak{z}_j}.
 \] 
Then analogously to (\ref{om1}--\ref{om3}) we deduce that the isomorphism of $H$-type structures is given by $\tilde{J}_{pz}=p\tilde{J}_{z}p^{-1}.$

Before we prove a general result on isomorphisms we look at our illustrative example for $k=2$.

\subsubsection{The case $k=2$}
Let us realize this scheme constructing a Lie algebra isomorphism $\mathfrak{n}_1\to\mathfrak{n}_3$ with the free Lie algebra isomorphism $B$ as follows. We define $B$ on $U$ as $$B\colon e_1\to e_2,\quad e_2\to -e_1,\quad e_3\to e_4,\quad \text{and \  $e_4\to e_3$}.$$
Then $B(e_i\times e_j):= Be_i\times Be_j$. Then, $\omega^{\Omega_3}_l=B(\omega^{\Omega_1}_l)$, and $\omega^{A_3}_l=B(\omega^{A_1}_l)$ for $l=1,2,3$. 
The corresponding matrix $B$ satisfies the orthogonality
condition $B^{-1}=B^T$ on $U$. The condition \eqref{om1} follows from the equality
\[
(B(e_a\times e_b), Be_i\times Be_j)_Z=(Be_a\times Be_b, Be_i\times Be_j)_Z=\frac{1}{k}\delta_{a,i}\delta_{b,j}=(e_a\times e_b, e_i\times e_j)_Z.
\]

Analogously, there exists an isomorphism $C$  between $\mathfrak{n}_1$ and $\mathfrak{n}_2$ acting as a signed basis permutation 
$$C\colon e_1\to e_1,\quad e_2\to e_2,\quad e_3\to e_4,\quad \text{and}\quad e_4\to e_3.$$
Correspondingly, $C(\omega^{\Omega_1}_1)=\omega^{\Omega_2}_1$, $C(\omega^{\Omega_1}_2)=\omega^{\Omega_2}_3$, and $C(\omega^{\Omega_1}_3)=\omega^{\Omega_2}_2$.

Thus,
all solutions give isomorphic $H$-type algebras $\mathfrak{n}$ with 3-dimensional center $\mathfrak{z}$ and 4-dimensional  space $\mathfrak{h}$. 

Of course, all 2-dimensional  centers can be realized 
by choosing $\mathfrak{z}$ to be generated by any two vectors $\omega$ from each of six solutions. In order to show isomorphisms
between all 18 options, it is sufficient to establish isomorphism between all possible three pairs in one of the solutions, and then, to compose 
with already established isomorphisms $B$ and $C$.
For example, for $\Omega_1$, we have 
\begin{itemize}
\item $(\omega^{\Omega_1}_1, \omega^{\Omega_1}_2)\simeq (\omega^{\Omega_1}_3, \omega^{\Omega_1}_2)$
by permutation $e_1\to e_1$,  $e_2\to -e_4$, $e_3\to e_3$, and  $e_4\to e_2$; 
\item $(\omega^{\Omega_1}_1, \omega^{\Omega_1}_2)\simeq (\omega^{\Omega_1}_1, \omega^{\Omega_1}_3)$
by permutation $e_1\to e_1$,  $e_2\to e_2$, $e_3\to -e_4$, and  $e_4\to e_3$.
\end{itemize}
Then we construct the orthogonal decompositions  $N=U\oplus_{\bot} \Xi_j\oplus_{\bot} F_j$, $j=1,2,3$, where $\Xi_1=\spn(\omega^{\Omega_1}_1, \omega^{\Omega_1}_2)$, $\Xi_2=\spn(\omega^{\Omega_1}_2, \omega^{\Omega_1}_3)$, $\Xi_1=\spn(\omega^{\Omega_1}_1, \omega^{\Omega_1}_3)$ and $F_j$ is the orthogonal complement in $Z$ to $\Xi_j$. The above signed basis permutations give isomorphisms between the $H$-type Lie algebras $N/F_j$.

Observe that dimension 1 of $\mathfrak{z}$ can be chosen always for any even dimension of $\mathfrak{h}$ and all such algebras will generate one and the same Heisenberg algebra.

\subsubsection{Existence of isomorphism}

Let us turn to the general case $k\geq 2$, and let us construct $B\colon N\to N$ satisfying the conditions at the beginning of Section~3.3. 
Let $\{E_l\}_{l=1}^{r}$ and $\{\tilde{E}_l\}_{l=1}^{r}$, $r=\rho(2k)-1$ be two different HR families of orthogonal matrices, i.e., solutions $E_l$ to the equations $E_j^2=-I$ and $E_iE_j=-E_jE_i$ with
the entries $0,\pm 1$, and let $B$ be a unitary signed basis permutation of $U$. Our method works for any $r\leq \rho(2k)-1$, but for simplicity we describe the case $r= \rho(2k)-1$. 

In order to find an isomorphism $p$ between $H$-type algebras we must find a signed basis permutation $B$ on $U$ satisfying 
 the system of $\rho(2k)-1$ equations  
$$
\tilde{E}_l=BE_l B^{-1}
$$ 
equivalent to \eqref{om3}.

We need the following technical combinatorial lemma.

\begin{lemma}\label{combA} Let $1, a_1,\dots, a_r$ be elements of a commutative ring with the multiplicative unity $1$ over  $\mathbb R$, such that $a_j^2=1$, $j=1,\dots,r$.
Then
\[
\left(r-\sum\limits_{j=1}^ra_j\right)\prod\limits_{j=1}^r(1+a_j)=0.
\]
\end{lemma}
\begin{proof}
We observe that $a_i\prod\limits_{j=1}^r(1+a_j)=\prod\limits_{j=1}^r(1+a_j)$. Therefore, $$\sum\limits_{j=1}^ra_j\prod\limits_{j=1}^r(1+a_j)=r \prod\limits_{j=1}^r(1+a_j).$$
\end{proof}

\begin{proposition}
Let $\{E_l\}_{l=1}^{r}$ and $\{\tilde{E}_l\}_{l=1}^{r}$, $r=\rho(2k)-1$ be two different HR families of orthogonal matrices with entries $0,\pm 1$. Then,  there exists a signed  permutations $B$ of the basis $U$, such that $\tilde{E}_l=BE_lB^{-1}$ for all $l=1,2,\dots r$.
\end{proposition}
\begin{proof} If $B$ is a signed  permutation of the standard basis of $\mathbb R^{2k}$, then $B^T=B^{-1}$ and $|\det B|=1$. Let $\tilde{E}_l=BE_lB^{-1}$.  We check the following properties
of an HR family of matrices.
\begin{itemize}
\item Obviously $\tilde{E}_l$ has entries $0,\pm 1$;
\item $\tilde{E}_l^{-1}=(BE_lB^{-1})^{-1}=BE_1^{-1}B^{-1}=BE_1^TB^{T}=(BE_lB^{T})^{T}=(BE_lB^{-1})^{T}=\tilde{E}_l^{T}$;
\item $\tilde{E}_l^{2}=(BE_lB^{-1})^{2}=BE_lB^{-1}BE_lB^{-1}=BE_l^2B^{-1}=-I$;
\item $\tilde{E}_i\tilde{E}_j=BE_iB^{-1}BE_jB^{-1}=BE_iE_jB^{-1}=-BE_jE_iB^{-1}=-\tilde{E}_j\tilde{E}_i$.
\end{itemize}
Therefore, the family of matrices $\{\tilde{E}_l\}_{l=1}^{\rho(2k)-1}$ is an HR-family for a fixed $B$. 

Reciprocally, given two HR-families $\{{E}_l\}_{l=1}^{\rho(2k)-1}$ and $\{\tilde{E}_l\}_{l=1}^{\rho(2k)-1}$,
we have to find a solution $B$ to the overdetermined system of equations  $\tilde{E}_l=BE_lB^{-1}$ for all $l=1,2,\dots ,r$.

Let us use the Kronecker matrix product  identity
\[
\text{vec}(ABC)=(C^T\otimes A)\text{vec}(B),
\]
where $\otimes$ stands for the {\it Kronecker matrix product} and vec$(V)$ for the {\it vectorization} of a matrix $V$, a linear transformation which converts the matrix into a column vector. 
The Kronecker matrix product possesses several nice properties, some necessary of them we list below.
\begin{itemize}
\item Bilinearity and associativity: 
\begin{itemize}
\item $A\otimes (B+C)=A\otimes B+A\otimes C$,
\item $(B+C)\otimes A= B\otimes A+C\otimes A$,
\item $(kA)\otimes B=A\otimes (kB)$, $k$ is a scalar,
\item $A\otimes (B\otimes C)=(A\otimes B)\otimes C$;
\end{itemize}
\item Transposition $(A\otimes B)^T=A^T\otimes B^T$;
\item For $r\times r$ matrices the determinant is $\det(A\otimes B)=(\det A)^r(\det B)^r$;
\item The mixed-product property $(A\otimes B)(C\otimes D)=AC\otimes BD$.
\end{itemize}

Observe that
$$
\tilde{E}_l=BE_l B^{-1}\,\Leftrightarrow \, \tilde{E}_lB=BE_l\,\Leftrightarrow\, \tilde{E}_lBI=IBE_l.
$$  
So the equation $\tilde{E}_lB=BE_l$ is transformed to the equation
\begin{equation}\label{Kr_eq}
(E_l^T\otimes I-I\otimes  \tilde{E}_l)\text{vec}(B)=0,
\end{equation}
which is a system of $(2k)^2$ equations for each $l$.

Let us analize the matrix of our system
\[
\mathcal{E}=\left(
\begin{array}{c}
E_1^T\otimes I-I\otimes  \tilde{E}_1\\
E_2^T\otimes I-I\otimes  \tilde{E}_2\\
\dots\\
E_r^T\otimes I-I\otimes  \tilde{E}_r
\end{array}\right).
\] 
In order to compute the rank of $\mathcal E$, we use the equality $\rank \mathcal E=\rank \mathcal E^T\mathcal E$, and calculate
\[
\mathcal F:=\mathcal E^T\mathcal E=\left( \left(E_1\otimes I-I\otimes  \tilde{E}_1^T,\dots, E_r\otimes I-I\otimes  \tilde{E}_r^T\right)\cdot \left(
\begin{array}{c}
E_1^T\otimes I-I\otimes  \tilde{E}_1\\
E_2^T\otimes I-I\otimes  \tilde{E}_2\\
\dots\\
E_r^T\otimes I-I\otimes  \tilde{E}_r
\end{array}\right)\right)
\]
\[
=(E_1\otimes I)(E_1^T\otimes I)+(I\otimes \tilde{E}_1^T)(I\otimes \tilde{E}_1)-(E_1\otimes I)(I\otimes \tilde{E}_1)-(I\otimes \tilde{E}_1^T)(E_1^T\otimes I)+\dots
\]
\[
+(E_r\otimes I)(E_r^T\otimes I)+(I\otimes \tilde{E}_r^T)(I\otimes \tilde{E}_r)-(E_r\otimes I)(I\otimes \tilde{E}_r)-(I\otimes \tilde{E}_r^T)(E_r^T\otimes I)=
\]
\[
=2\sum\limits_{j=1}^r(I-(E_j\otimes \tilde{E}_j))=2\sum\limits_{j=1}^r((I\otimes I)-(E_j\otimes \tilde{E}_j)),
\]
where we distinguish the dimension of the matrix $I$ in $\mathcal F$ and in its blocks. We use that $(E_i\otimes \tilde{E}_i)(E_j\otimes \tilde{E}_j)=(E_iE_j)\otimes (\tilde{E}_i\tilde{E}_j)=(E_jE_i)\otimes (\tilde{E}_j\tilde{E}_i)$. Notice that  if we denote $A_i=E_i\otimes \tilde{E}_i$, then
$A_i^2=I$, $A_iA_j=A_jA_i$, and they are elements of a commutative ring and satisfy
 Lemma~\ref{combA}. So we write
\[
\mathcal F(I+A_1+\dots +A_r)\equiv 2(rI-A_1-\dots -A_r)\prod\limits_{i=1}^r(I+A_j)=0.
\]
The matrix $I+A_1+\dots +A_r$ is non-vanishing (it has all 1 in the principal diagonal), therefore, $\det \mathcal F=0$, and the rank
of the matrix $\mathcal E$ is smaller than $(2k)^2$. Thus, a solution $B$ exists. We also observe that the matrix $\mathcal{E}$ in the left hand side of equation~\eqref{Kr_eq} contains only two non-zero entries in each row and column and these entries equal $\pm 1$. Thus the system is reduces to the form $b_{ij}\pm b_{ab}=0$ and the orthogonality $B^T=B^{-1}$ follows from the form the equation. Therefore, choosing free variables $b_{ij}$ equal $0$, $\pm 1$, we always can find a solution $B$ consisting of $0$, $\pm 1$.
\end{proof}

We visualise this method in our example. Denote by $E_l^{\Omega_j}$ the matrix which corresponds to the vector $\omega_l^{\Omega_j}$.
Then we choose, for instance, $\tilde{E}_1=E_1^{\Omega_3}=BE_1^{\Omega_1}B^{-1}$ and write equation~\eqref{Kr_eq}

{\tiny
\begin{equation}\label{eq_b}
\renewcommand{\arraystretch}{1.4}
\left(\begin{array}{rrrr|rrrr|rrrr|rrrr}
0&1&0&0&  1&0&0&0& 0&0&0&0& 0&0&0&0 \\
-1&0&0&0&  0&1&0&0& 0&0&0&0& 0&0&0&0 \\
0&0&0&1&  0&0&1&0& 0&0&0&0& 0&0&0&0 \\
0&0&-1&0&  0&0&0&1& 0&0&0&0& 0&0&0&0 \\
 \hline
-1&0&0&0&  0&1&0&0& 0&0&0&0& 0&0&0&0 \\
0&-1&0&0&  -1&0&0&0& 0&0&0&0& 0&0&0&0 \\
0&0&-1&0&  0&0&0&1& 0&0&0&0& 0&0&0&0 \\
0&0&0&-1&  0&0&-1&0& 0&0&0&0& 0&0&0&0 \\
 \hline
0&0&0&0&  0&0&0&0& 0&1&0&0& -1&0&0&0 \\
0&0&0&0&  0&0&0&0& -1&0&0&0& 0&-1&0&0 \\
0&0&0&0&  0&0&0&0& 0&0&0&1& 0&0&-1&0 \\
0&0&0&0&  0&0&0&0& 0&0&-1&0& 0&0&0&-1 \\
 \hline
0&0&0&0&  0&0&0&0& 1&0&0&0& 0&1&0&0 \\
0&0&0&0&  0&0&0&0& 0&1&0&0& -1&0&0&0 \\
0&0&0&0&  0&0&0&0& 0&0&1&0& 0&0&0&1 \\
0&0&0&0&  0&0&0&0& 0&0&0&1& 0&0&-1&0 
  \end{array}\right) 
 \cdot \left(\begin{array}{c}
b_{11}\\
b_{12}\\
b_{13}\\
b_ {14}\\
  \hline
 b_{21}\\
b_{22}\\
b_{23}\\
b_ {24}\\
\hline
b_{31}\\
b_{32}\\
b_{33}\\
b_ {34}\\
  \hline
  b_{41}\\
b_{42}\\
b_{43}\\
b_ {44}\\
 \end{array}\right)=0.
\end{equation}
}
The matrix $B$ which is solution of~\eqref{eq_b} is the following
{\tiny
\begin{equation*}
\renewcommand{\arraystretch}{1.4}
\left(\begin{array}{rrrr|rrrr|rrrr|rrrr}
0&1&0&0&  1&0&0&0& 0&0&0&0& 0&0&0&0 \\
-1&0&0&0&  0&1&0&0& 0&0&0&0& 0&0&0&0 \\
0&0&0&1&  0&0&1&0& 0&0&0&0& 0&0&0&0 \\
0&0&-1&0&  0&0&0&1& 0&0&0&0& 0&0&0&0 \\
 \hline
-1&0&0&0&  0&1&0&0& 0&0&0&0& 0&0&0&0 \\
0&-1&0&0&  -1&0&0&0& 0&0&0&0& 0&0&0&0 \\
0&0&-1&0&  0&0&0&1& 0&0&0&0& 0&0&0&0 \\
0&0&0&-1&  0&0&-1&0& 0&0&0&0& 0&0&0&0 \\
 \hline
0&0&0&0&  0&0&0&0& 0&1&0&0& -1&0&0&0 \\
0&0&0&0&  0&0&0&0& -1&0&0&0& 0&-1&0&0 \\
0&0&0&0&  0&0&0&0& 0&0&0&1& 0&0&-1&0 \\
0&0&0&0&  0&0&0&0& 0&0&-1&0& 0&0&0&-1 \\
 \hline
0&0&0&0&  0&0&0&0& 1&0&0&0& 0&1&0&0 \\
0&0&0&0&  0&0&0&0& 0&1&0&0& -1&0&0&0 \\
0&0&0&0&  0&0&0&0& 0&0&1&0& 0&0&0&1 \\
0&0&0&0&  0&0&0&0& 0&0&0&1& 0&0&-1&0 
  \end{array}\right) 
 \cdot \left(\begin{array}{c}
0\\
1\\
0\\
0\\
  \hline
 -1\\
0\\
0\\
0\\
\hline
0\\
0\\
0\\
1\\
  \hline
 0\\
0\\
1\\
0\\
 \end{array}\right)=0.
\end{equation*}
}
Check 
$E_2^{\Omega_3}=BE_2^{\Omega_1}B^{-1}$
\[
\renewcommand{\arraystretch}{1.4}
\left(\begin{array}{rrrr}
0&0&1&0   \\
0&0&0&-1   \\
-1&0&0&0   \\
0&1&0&0  
  \end{array}\right) =\left(\begin{array}{rrrr}
0&-1&0&0   \\
1&0&0&0   \\
0&0&0&1   \\
0&0&1&0  
  \end{array}\right)\left(\begin{array}{rrrr}
0&0&-1&0   \\
0&0&0&-1   \\
1&0&0&0   \\
0&1&0&0  
  \end{array}\right)\left(\begin{array}{rrrr}
0&1&0&0   \\
-1&0&0&0   \\
0&0&0&1   \\
0&0&1&0  
  \end{array}\right),
\]
and
$E_3^{\Omega_3}=BE_3^{\Omega_1}B^{-1}$
\[
\renewcommand{\arraystretch}{1.4}
\left(\begin{array}{rrrr}
0&0&0&1   \\
0&0&1&0   \\
0&-1&0&0   \\
-1&0&0&0  
  \end{array}\right) =\left(\begin{array}{rrrr}
0&-1&0&0   \\
1&0&0&0   \\
0&0&0&1   \\
0&0&1&0  
  \end{array}\right)\left(\begin{array}{rrrr}
0&0&0&1   \\
0&0&-1&0   \\
0&1&0&0   \\
-1&0&0&0 
  \end{array}\right)\left(\begin{array}{rrrr}
0&1&0&0   \\
-1&0&0&0   \\
0&0&0&1   \\
0&0&1&0  
  \end{array}\right).
\]
We see that the matrix
\[
B_1=\left(\begin{array}{rrrr}
1&0&0&0   \\
0&1&0&0   \\
0&0&0&1   \\
0&0&1&0  
  \end{array}\right)
\] 
gives also a solution to the equation $E_1^{\Omega_2}=B_1E_1^{\Omega_1}B_1^{-1}$ which is the same as \eqref{eq_b} because $\omega_1^{\Omega_2}=\omega_1^{\Omega_3}$:
{\tiny
\begin{equation*}
\renewcommand{\arraystretch}{1.4}
\left(\begin{array}{rrrr|rrrr|rrrr|rrrr}
0&1&0&0&  1&0&0&0& 0&0&0&0& 0&0&0&0 \\
-1&0&0&0&  0&1&0&0& 0&0&0&0& 0&0&0&0 \\
0&0&0&1&  0&0&1&0& 0&0&0&0& 0&0&0&0 \\
0&0&-1&0&  0&0&0&1& 0&0&0&0& 0&0&0&0 \\
 \hline
-1&0&0&0&  0&1&0&0& 0&0&0&0& 0&0&0&0 \\
0&-1&0&0&  -1&0&0&0& 0&0&0&0& 0&0&0&0 \\
0&0&-1&0&  0&0&0&1& 0&0&0&0& 0&0&0&0 \\
0&0&0&-1&  0&0&-1&0& 0&0&0&0& 0&0&0&0 \\
 \hline
0&0&0&0&  0&0&0&0& 0&1&0&0& -1&0&0&0 \\
0&0&0&0&  0&0&0&0& -1&0&0&0& 0&-1&0&0 \\
0&0&0&0&  0&0&0&0& 0&0&0&1& 0&0&-1&0 \\
0&0&0&0&  0&0&0&0& 0&0&-1&0& 0&0&0&-1 \\
 \hline
0&0&0&0&  0&0&0&0& 1&0&0&0& 0&1&0&0 \\
0&0&0&0&  0&0&0&0& 0&1&0&0& -1&0&0&0 \\
0&0&0&0&  0&0&0&0& 0&0&1&0& 0&0&0&1 \\
0&0&0&0&  0&0&0&0& 0&0&0&1& 0&0&-1&0 
  \end{array}\right) 
 \cdot \left(\begin{array}{c}
1\\
0\\
0\\
0\\
  \hline
0\\
1\\
0\\
0\\
\hline
0\\
0\\
0\\
1\\
  \hline
 0\\
0\\
1\\
0\\
 \end{array}\right)=0
\end{equation*}
}
Check 
$E_3^{\Omega_2}=B_1E_2^{\Omega_1}B_1^{-1}$
\[
\renewcommand{\arraystretch}{1.4}
\left(\begin{array}{rrrr}
0&0&0&-1   \\
0&0&-1&0   \\
0&1&0&0   \\
1&0&0&0  
  \end{array}\right) =\left(\begin{array}{rrrr}
1&0&0&0   \\
0&1&0&0   \\
0&0&0&1   \\
0&0&1&0  
  \end{array}\right)\left(\begin{array}{rrrr}
0&0&-1&0   \\
0&0&0&-1   \\
1&0&0&0   \\
0&1&0&0  
  \end{array}\right)\left(\begin{array}{rrrr}
1&0&0&0   \\
0&1&0&0   \\
0&0&0&1   \\
0&0&1&0  
  \end{array}\right),
\]
and
$E_2^{\Omega_3}=B_1E_3^{\Omega_1}B_1^{-1}$
\[
\renewcommand{\arraystretch}{1.4}
\left(\begin{array}{rrrr}
0&0&1&0   \\
0&0&0&-1   \\
-1&0&0&0   \\
0&1&0&0  
  \end{array}\right) =\left(\begin{array}{rrrr}
1&0&0&0   \\
0&1&0&0   \\
0&0&0&1   \\
0&0&1&0  
  \end{array}\right)\left(\begin{array}{rrrr}
0&0&-1&0   \\
0&0&0&-1   \\
1&0&0&0   \\
0&1&0&0  
  \end{array}\right)\left(\begin{array}{rrrr}
1&0&0&0   \\
0&1&0&0   \\
0&0&0&1   \\
0&0&1&0  
  \end{array}\right).
\]
We see that different basis permutation matrices  $B$ and $B_1$ which are the solutions to \eqref{eq_b} give isomorphisms of $H$-type algebras.

\subsection{Non-degenerate metric}

In this section we study the case when the metric $(\cdot\,,\cdot)_{U}$ is neutral,  $(\cdot\, ,\cdot)_{Z}$ is non-positive, and both metrics are non-degenerate, see Proposition~\ref{pr4}. 

\subsubsection{Construction of pseudo $H$-type algebras} 
We need to prove an analogue of Theorem~\ref{t6}, and we need auxiliary technical lemmas. Let us start reminding that for the inner (positive definite) product, if  $\omega_{\ell}=\dots+\alpha^{\ell}_{ij}(e_i\times e_j)+\dots$ for some  $\mathfrak{p}_{\ell}$, where $i<j$, then $J_{\omega_{\ell}}e_i=\alpha^{\ell}_{ij} e_j$, where $\alpha^{\ell}_{ij}$ is (+1) or (-1). In the case of a non-degenerate scalar product the situation is more complicated, and we describe it in the following proposition.    For definiteness, we assume that $(e_i, e_i)_{U}=1/k$ for $i=1,2,\dots,k$ and $(e_i, e_i)_{U}=-1/k$ for $i=k+1,k+2,\dots,2k$.

\begin{proposition}\label{p6} Let $\omega_{\ell}=\sum\limits_{(i,j)\in \mathfrak{p}_{\ell}}\alpha^{\ell}_{ij}(e_i\times e_j)$, $\mathfrak{p}_{\ell}\in\mathfrak{p}'$, $\alpha^{\ell}_{ij}=\pm 1$.
Denote by \linebreak $\beta:=k\alpha_{ij}^{\ell}(e_i\times e_j,e_i\times e_j)_{Z}$ for an arbitrary chosen pair $(i,j)\in \mathfrak{p}_{\ell}$. Then
\begin{itemize}
\item If  $1 \leq i<j\leq k$, then $J_{\omega_{\ell}}e_i=\beta e_j$, and $J_{\omega_{\ell}}e_j=-\beta e_i$.
\item  If  $k+1 \leq i<j\leq 2k$, then $J_{\omega_{\ell}}e_i=-\beta e_j$, and $J_{\omega_{\ell}}e_j=\beta e_i$.
\item  If $1 \leq i\leq k$ and  $k+1 \leq j\leq 2k$, then $J_{\omega_{\ell}}e_i=-\beta e_j$, and $J_{\omega_{\ell}}e_j=-\beta e_i$.
\end{itemize}
\end{proposition}
\begin{proof}
Recall that $J_{\omega_{\ell}}(e_i)=\pm e_j$, $(i,j)\in \mathfrak{p}_{\ell}$. 
Let $1 \leq i<j\leq k$.  
Then using the definition of the operator $J_{\omega_{\ell}}$ we obtain
\begin{equation}\label{from}
(J_{\omega_{\ell}}e_i, e_j)_{U}=(\omega_{\ell}, e_i\times e_j)_{Z}=\alpha_{ij}^{\ell}(e_i\times e_j,e_i\times e_j)_{Z}=\beta\frac{1}{k}=\beta(e_j, e_j)_{U}.
\end{equation}
The integer $i$ appears in pairs of $\mathfrak{p}_{\ell}$ only in the pair $(i,j)$, therefore the equality $$(J_{\omega_{\ell}}e_i, e_m)_{U}=\beta(e_j, e_m)_{U}=0$$ trivially holds for all $m=1,\dots, 2k$, $m\neq j,$  hence $J_{\omega_{\ell}}e_i=\beta e_j$ as in the formulation of Proposition~\ref{p6}.
At the same time, $(e_j, e_j)_{U}=(e_i, e_i)_{U}$, and by skew symmetry of $J_{\omega_{\ell}}$ we have $(J_{\omega_{\ell}}e_i, e_j)_{U}=-(J_{\omega_{\ell}}e_j, e_i)_{U}$. So 
$
(J_{\omega_{\ell}}e_j, e_i)_{U}=-\beta(e_i, e_i)_{U},
$
from \eqref{from},
which implies $J_{\omega_{\ell}}e_j=-\beta e_i$.
Analogously we argue for $k+1 \leq i,j\leq 2k$. 

For $1 \leq i\leq k$ and  $k+1 \leq j\leq 2k$, arguing as in~\eqref{from}, we obtain
$
(J_{\omega_{\ell}}e_i, e_j)_{U}=-\beta(e_j, e_j)_{U}$.
Since $(e_j, e_j)_{U}=-(e_i, e_i)_{U}$, and 
$
(J_{\omega_{\ell}}e_j, e_i)_{U}=-\beta(e_i, e_i)_{U},
$ by skew symmetry of $J_{\omega_{\ell}}$, we proof the third statement of Proposition~\ref{p6}.
\end{proof}
Both, $\beta$ and $\alpha_{ij}^{\ell}$, admit values in $\{-1,1\}$, observe that $|(e_i\times e_j,e_i\times e_j)_{\mathfrak{z}}|=\frac{1}{k}$, and the choice of the sign of $\beta$ depends both on $\alpha_{ij}^{\ell}$ and on the signature of $(\cdot\, , \cdot)_{Z}$.

Assume that the family of 1-factors $\mathfrak{p}''\subseteq \mathfrak{p}'$ satisfies the property that any two factors set together form the union of disjoint cycles. 

\begin{proposition}\label{p7} If $\mathfrak{p}_{\ell}\in \mathfrak{p}'$ defines $\omega_{\ell}$ satisfying the condition $J^2_{\omega_{\ell}}=-(\omega_{\ell},\omega_{\ell})_{Z} \id_U$, then
 we have two options. The pairs $(i,j)\in \mathfrak{p}_{\ell}$ are of the following types:
\begin{itemize}
\item either $1 \leq i,j\leq k$ or $k+1 \leq i,j\leq 2k$ for all $(i,j)\in \mathfrak{p}_{\ell}$;
\item $1 \leq i\leq k$ and $k+1 \leq j\leq 2k$  for all $(i,j)\in \mathfrak{p}_{\ell}$.
\end{itemize}
\end{proposition}
\begin{proof}
Indeed, assume that for  $\mathfrak{p}_{\ell}$, there exist two pairs  $(i_1,j_1)\in \mathfrak{p}_{\ell}$ and $(i_2,j_2)\in \mathfrak{p}_{\ell}$, such that
$1 \leq i_1,j_1\leq k$ but $1 \leq i_2\leq k$ and $k+1 \leq j_2\leq 2k$. Specify indices for $\beta$ from Proposition~\ref{p6} as $\beta_{{ij}}:=k\alpha_{ij}^{\ell}(e_i\times e_j,e_i\times e_j)_{Z}$. Then
\[
J_{\omega_{\ell}}e_{i_1}=\beta_{i_1j_1} e_{j_1}, \quad J^2_{\omega_{\ell}}e_{i_1}=-\beta_{i_1j_1}^2e_{i_1}=-e_{i_1},
\]
by Proposition~\ref{p6}, and necessarily  $(\omega_{\ell},\omega_{\ell})_{\mathfrak{z}}=1$. But 
\[
J_{\omega_{\ell}}e_{i_2}=-\beta_{i_2j_2} e_{j_2}, \quad J^2_{\omega_{\ell}}e_{i_2}=\beta_{i_2j_2}^2e_{i_2}=e_{i_1},
\]
and $(\omega_{\ell},\omega_{\ell})_{\mathfrak{z}}=-1$ at the same time. This contradiction brings us to the  statement. The same proof holds for $k+1 \leq i_1,j_1\leq 2k$.
\end{proof}

This, for example, means that the partition  $\mathfrak{p}_{\ell}=\{(1,2), (3,4), (5,6), (7,8)\}$ satisfies Proposition~\ref{p7}  but $\mathfrak{p}_{\ell}=\{(1,2), (3,7), (4,8), (5,6)\}$ for $k=4$ does not. f Proposition~\ref{p7} yields that the operators $J_{\omega_{\ell}}$ have a matrix representation in the block form
 \[
E_{\ell}=\left(
\begin{array}{cc} A & B \\ 
C & D 
\end{array}
\right),
\]
where the matrices $A$, $D$ are skew-symmetric, and $B=C^T$. For example, in the case $k=2$, and the signature $(1,2)$ we have
\[
E_1=\left(
\begin{array}{rrrr} 0 & \alpha_{12} & 0 & 0 \\ 
-\alpha_{12} & 0 & 0 & 0 \\ 
0 & 0 & 0 & \alpha_{34} \\
0 & 0 &  -\alpha_{34}  & 0 \\  
\end{array}
\right),\quad E_2=\left(
\begin{array}{rrrr} 0 & 0 & \alpha_{13} & 0 \\ 
0 & 0 & 0 & \alpha_{24} \\ 
\alpha_{13} & 0 & 0 & 0 \\
0 & \alpha_{24} &  0  & 0 \\  
\end{array}
\right),
\]
\[
E_3=\left(
\begin{array}{rrrr} 0 & 0 & 0 & \alpha_{14} \\ 
0 & 0 & \alpha_{23} & 0 \\ 
0 & \alpha_{23} & 0 & 0 \\
\alpha_{14} & 0 &  0  & 0 \\  
\end{array}
\right),
\]
Proposition \ref{p7} has a clear meaning in the matrix form, i.e., all non-zero elements $\alpha_{ij}$, $i,j\in \mathfrak{p}_{\ell}$ are either in $A$ and $D$ or in $B$ and $C$ blocks
of the matrix $E_{\ell}$.

\subsubsection{Construction of partitions} 
An analogue of Theorem~\ref{Eckmann} for generalized HR family of type $(s,t)$ was proved by Wolfe~\cite{Wolfe} in 1976.
Let us define a family of real matrices HR$(s,t)$ by the following conditions
\begin{itemize}
\item $E_0=I$, $E^2_k=-I$ for $k=1,\dots, s$; $E^2_k=I$ for $k=s+1,\dots, s+t$;
\item $E_k=-E^T_k$ for $k=1,\dots, s$; $E_k=E^T_k$ for for $k=s+1,\dots, s+t$;
\item $E_kE_j+E_jE_k=0$ for $k\neq j$ and $k,j=1,\dots, s+t$.
\end{itemize}

\begin{definition}
Let $n$, $s$  be positive integers and $t\geq 0$. We define
\[
\rho_t(n)=\max\{s\colon \text{$Cl_{s-1,t}$ has an irreducible ($n\times n$) matrix representation over $\mathbb R$}\},
\]
\[
\sigma_s(n)=\max\{t\colon \text{$Cl_{s,t-1}$ has an irreducible  ($n\times n$)  matrix representation over $\mathbb R$}\}.
\]
\end{definition} 

Wolfe \cite{Wolfe} completely characterised $\rho_t(n)$ and $\sigma_s(n)$. In particular, 
\begin{itemize}
\item $\rho_1(2)=\rho_2(2)=2$, $\rho_5(8)=1$;
\item $\rho_t(2n)=\rho_{t-1}(n)+1$;
\item $\rho_t(n)=\rho_{t+8}(16 n)$;
\end{itemize}

\begin{itemize}
\item $\sigma_1(2)=3$, $\sigma_3(4)=\sigma_5(8)=\sigma_6(8)=\sigma_7(8)=1$;
\item $\sigma_s(2n)=\sigma_{s-1}(n)+1$;
\item $\sigma_s(n)=\sigma_{s+8}(16 n)$;
\end{itemize}

Both $\rho_t(n)$ and $\sigma_s(n)$ are completely characterised by the Hurwitz-Radon function $\rho(n)$ and the above relations as $\rho_0(n)=\rho(n)$ and $\sigma_0(2n)=\rho(n)+2$. 
Define the next function $\tau(n)=\max\{\rho_t(n)+t\colon t\geq 0\}$. If $n=2^r$, then $\tau(n)=2(r+1)$.
Given positive integer $n$, if there exists an HR$(s,t)$ family in order $n$, then~\cite{Wolfe},
\begin{itemize}
\item $s+t\leq\tau(n)-1$;
\item $s\leq \rho_t(n)-1$ and $t\leq \sigma_s(n)-1$.
\end{itemize}
Moreover, the limits are achievable and the matrices can be assumed with integer entries. 
The dimension of a Clifford algebra is a power of $2$ and $\rho_t(u2^{r})=\rho_t(2^{r})$ for odd $u$, hence it is sufficient to reduce the dimension of the representation space
to the closest power of $2$.
In order to come up with an analogue of Theorem~\ref{t6} we need the following result proved by Ciatti in \cite{Ciatti}. 
\begin{theorem}[\cite{Ciatti}]\label{TCiatti}
The following triplets $(s,t,r)$ guarantee that the Clifford algebra $Cl_{s,t}$ has irreducible and admissible representation over the representation space of minimal dimension $2^{r}$.
\begin{itemize}
\item $s-t\equiv 0,\, 6${\rm (mod 8)} for $r\equiv 0,\, 6${\rm (mod 4)};
\item $s-t\equiv 2,\, 4${\rm (mod 8)} for  all values of $r$;
\item $s-t\equiv 1,\, 5${\rm (mod 8)} for $r\equiv 0,\, 2,\, 3${\rm (mod 4)};
\item $s-t\equiv 3${\rm (mod 8)} for $r\equiv 0,\, 2${\rm (mod 4)};
\item $s-t\equiv 7${\rm (mod 8)} for $r\equiv 3${\rm (mod 4)}.
\end{itemize}
\end{theorem}

All irreducible representations of Clifford algebras $Cl_{s,0}$ are admissible with an inner product, and the representation space is of dimension $2^{r}$, where $s=\rho(2^r)-1$.
In the case $t>0$, if $(s,t)$ does not satisfy the conditions of Theorem~\ref{TCiatti}, then we need to double the representation space and the corresponding representation map has to be redefined in a proper way, see \cite{Ciatti}.
 
Modifying reasonings given in the previous section and applying the Wolfe's existence theorem \cite[Theorem 2.8]{Wolfe} we can  formulate the following theorem.

\begin{theorem}\label{te6}
Let $N=U\oplus_{\perp}Z$ be a free nilpotent  Lie algebra  with a scalar product $(\cdot,\, \cdot)_N=(\cdot,\, \cdot)_U+ (\cdot,\, \cdot)_Z$ which is neutral  on $U$ and of an arbitrarily fixed signature $(s,t)$, $t\geq 1$, on $Z$, with dim$(U)=2^r$, $r\geq 1$. There exists an ideal $A$ such that $N/A$ is an $H$-type algebra $\mathfrak{n}=\mathfrak{h}\oplus_{\perp} \mathfrak{z}$, where $\mathfrak{h}=U$ and the center is $\mathfrak{z}=Z/A$. Given an integer $t\geq 1$, the ideal $A$ can be chosen such that $\dim(\mathfrak{z})=s+t$ and the signature $(s,t)$ are such that $s$ takes every value from 
\begin{itemize}
\item $0\leq s\leq \rho_t(2^r)-1$ in the case when the triplet $(\rho_t(2^r)-1, t, r)$ satisfies one of the conditions of Theorem~\ref{TCiatti};
\item Otherwise, $0\leq s\leq \rho_t(2^{r})-m(r,t)$, where $m=m(r,t)$, $1<m\leq \rho_t(2^{r})$ is the maximal positive integer for which either the triplet $(\rho_t(2^r)-m, t, r)$ satisfies  the conditions of Theorem~\ref{TCiatti} or $\rho_t(2^r)-m=\rho_t(2^{r-1})-1$.
\end{itemize}
\end{theorem}
\begin{proof} 
Similarly to the proof of Theorem~\ref{t6} we consider a family of partitions $$\mathfrak{p}'=\{\mathfrak{p}'_{1},\dots, \mathfrak{p}'_{2k-1}\},$$ equivalent to a  1-factorization of
the complete graph $K_{2k}$. We recall the  isomorphism between the operators $J_{\omega_l}$, $l\in 1,\dots, 2k-1$, and the orthogonal matrices $E_l$ constructed in the proof of Theorem~\ref{t6}.
Next we choose a family of partitions $\mathfrak{p}''=\{\mathfrak{p}''_{1},\dots, \mathfrak{p}''_{s+t}\}\subset \mathfrak{p}'$ such that the set of matrices $$\{E_{l_1},\dots E_{l_{s}},E_{l_{s+1}},\dots, E_{l_{s+t}}\}$$ is an  HR$(s,t)$ family  with the maximal value of $s$.

Now we set $\mathfrak{h}=U$ and $A$ to be the orthogonal complement in $Z$ to $$\spn\{\omega_{l_1},\dots, \omega_{l_{s}},\omega_{l_{s+1}},\dots, \omega_{l_{s+t}} \}.$$ Due to the correspondence between the operators $J_{\omega_{l_m}}$, $m=1,\dots,s+t$ and the matrices from the HR$(s,t)$ family, the operators $J_{\omega_{l_m}}$ satisfy the conditions of Proposition~\ref{propOrt}, and therefore, the orthogonality condition~\eqref{p2}.   Thus, we define $\mathfrak{z}=Z/A$, and $\mathfrak{h}\oplus_{\perp} \mathfrak{z}$ is a pseudo $H$-type algebra $\mathfrak{n}$.
\end{proof}
We illustrate Theorem~\ref{te6} in Figure 2 for all possible triplets $(s, t=t_{\rm fix}, r)$ satisfying  the conditions of Theorem~\ref{TCiatti}, which we denote by ($\star$) in the diagrams.
The first line corresponds to the situation when the triplet $(\rho_t(2^r)-1, t, r)$ satisfies one of the conditions of Theorem~\ref{TCiatti}, and then, the maximal value of $s$ is chosen. Otherwise, we analize one step down $s=\rho_t(2^r)-2$.
Then the representation space can be either of dimension $2^{r-1}$ (line 2 and 3) or of dimension $2^{r}$ (line 3 and 4). In the first case (line 2), we may have that the triplet $(\rho_t(2^r)-2, t, r-1)$ satisfies  the conditions of Theorem~\ref{TCiatti} and we choose the maximal value of $s$ to be $\rho_t(2^r)-2$. Otherwise (line 3), the triplet $(\rho_t(2^r)-2, t, r-1)$ does not satisfy  the conditions of Theorem~\ref{TCiatti}, however $\rho_t(2^r)-2=\rho_t(2^{r-1})-1$ and we choose the maximal value of $s$ to be $\rho_t(2^r)-2$. If  the representation space  for $s=\rho_t(2^r)-2$ is  of dimension $2^{r}$, then we may have the triplet $(\rho_t(2^r)-2, t, r)$ satisfying  the conditions of Theorem~\ref{TCiatti}, and then we choose the maximal value of $s$ to be $\rho_t(2^r)-2$.
Otherwise, we continue to step down $s\to s-1$ and repeat the previous procedure over again. Positivity of $s$ guarantees finiteness of this algorithm. Normally, it takes at most 3 steps
to finish the choice of the maximal value of $s$.

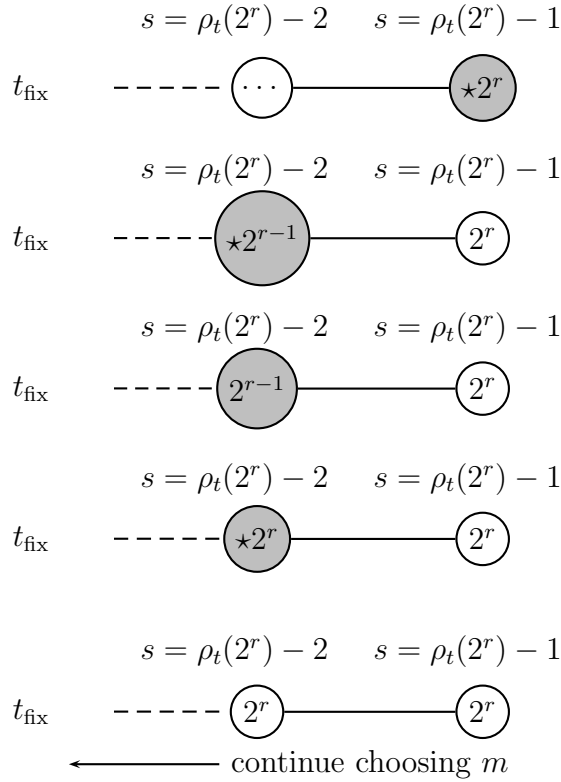
\begin{figure}
\begin{pspicture}(0,-1)(5,10)
\rput(5.5,0.2){\circlenode{A}{$2^r$}}
\rput(2.5,0.2){\circlenode{B}{$2^r$}}
\ncline{A}{B}
\psline[linestyle=dashed](0.6,0.2)(2,0.2)
\rput(2.2,1){$s= \rho_t(2^r)-2$}
\rput(5.3,1){$s= \rho_t(2^r)-1$}
\rput(-0.5,0.2){$t_{\rm fix}$}
\rput(5.5,2.5){\circlenode{A}{$2^r$}}
\rput(2.5,2.5){\circlenode[fillstyle=solid, fillcolor=lightgray]{B}{$\star 2^r$}}
\ncline{A}{B}
\psline[linestyle=dashed](0.6,2.5)(2,2.5)
\rput(2.2,3.3){$s= \rho_t(2^r)-2$}
\rput(5.3,3.3){$s= \rho_t(2^r)-1$}
\rput(-0.5,2.5){$t_{\rm fix}$}
\rput(5.5,4.5){\circlenode{A}{$2^{r}$}}
\rput(2.5,4.5){\circlenode[fillstyle=solid, fillcolor=lightgray]{B}{$2^{r-1}$}}
\ncline{A}{B}
\psline[linestyle=dashed](0.6,4.5)(1.9,4.5)
\rput(2.2,5.3){$s= \rho_t(2^r)-2$}
\rput(5.3,5.3){$s= \rho_t(2^r)-1$}
\rput(-0.5,4.5){$t_{\rm fix}$}
\rput(5.5,6.5){\circlenode{A}{$ 2^{r}$}}
\rput(2.5,6.5){\circlenode[fillstyle=solid, fillcolor=lightgray]{B}{$ \star 2^{r-1}$}}
\ncline{A}{B}
\psline[linestyle=dashed](0.6,6.5)(1.9,6.5)
\rput(2.2,7.4){$s= \rho_t(2^r)-2$}
\rput(5.3,7.4){$s= \rho_t(2^r)-1$}
\rput(-0.5,6.5){$t_{\rm fix}$}
\rput(5.5,8.5){\circlenode[fillstyle=solid, fillcolor=lightgray]{A}{$\star 2^{r}$}}
\rput(2.5,8.5){\circlenode{B}{$\dots$}}
\ncline{A}{B}
\psline[linestyle=dashed](0.6,8.5)(2,8.5)
\rput(2.2,9.4){$s= \rho_t(2^r)-2$}
\rput(5.3,9.4){$s= \rho_t(2^r)-1$}
\rput(-0.5,8.5){$t_{\rm fix}$}
\psline{<-}(0,-0.5)(2,-0.5)
\rput(4,-0.5){continue choosing $m$}
\end{pspicture}\label{fig1}
\caption{The diagram illustrates the algorithm of the choice of maximal $s$ in Theorem~\ref{te6}. Nodes mean the Clifford modules  of dimension written inside. Filled nodes mean the chosen position of $s$.}
\end{figure}

Analogously we formulate a theorem of existence of $H$-type algebras for a given value of~$s$.

\begin{theorem}\label{te7}
Let $N=U\oplus_{\perp}Z$ be a free nilpotent  Lie algebra  with a scalar product $(\cdot,\, \cdot)_N=(\cdot,\, \cdot)_U+ (\cdot,\, \cdot)_Z$ which is neutral  on $U$ and of arbitrarily fixed signature $(s,t)$, $t\geq 1$, on $Z$, with dim$(U)=2^r$, $r\geq 1$. There exists an ideal $A$ such that $N/A$ is an $H$-type algebra $\mathfrak{n}=\mathfrak{h}\oplus_{\perp} \mathfrak{z}$, where $\mathfrak{h}=U$ and the center is $\mathfrak{z}=Z/A$. Given an integer $s\geq 0$, the ideal $A$ can be chosen such that $\dim(\mathfrak{z})=s+t$ and the signature $(s,t)$ are such that $t$  takes every value from 
\begin{itemize}
\item $0\leq t\leq \sigma_s(2^r)-1$ in the case the triplet $(s, \sigma_s(2^r)-1, r)$ satisfies one of the conditions of Theorem~\ref{TCiatti};
\item Otherwise, $0\leq t\leq \sigma_s(2^r)-m$, where $1<m\leq \sigma_s(2^r)$ is the maximal positive integer for which either the triplet $(s, \sigma_s(2^r)-m, r)$ satisfies the conditions of Theorem~\ref{TCiatti} or $\sigma_s(2^r)-m= \sigma_s(2^{r-1})-1$.
\end{itemize}
\end{theorem}

Below we collect in Table~1 the first maximal values of the parameter $s$ with a given dimension of $\mathfrak{h}$ and a given value of $t$; and in Table~2 the first maximal values of the parameter $t$ with a given dimension of $\mathfrak{h}$ and a given value of $s$ such that the $H$-type algebra 
$\mathfrak{n}=\mathfrak{h}\oplus \mathfrak{z}$ exists.

\begin{table}[ht!]
  \centering
 \begin{tabular}{|l||*{5}{c|}}\hline
\backslashbox{$\,\,\,2^r$}{$t\,\,\,$}
&\makebox[3em]{0}&\makebox[3em]{1}&\makebox[3em]{2}
&\makebox[3em]{3}&\makebox[3em]{4}\\\hline\hline
$2$ &1&0&-&-&-\\\hline
$4$ &3&1&1&-&-\\\hline
$8$ &7&3&3&3&3\\\hline
$16$ &8&7&5&4&4\\\hline
$32$ &9&8&7&5&5\\\hline
\end{tabular}
\vspace{15pt}
\caption{Maximal values of $s$}
 \end{table}

\begin{table}[ht!]
  \centering
 \begin{tabular}{|l||*{5}{c|}}\hline
\backslashbox{$\,\,\,2^r$}{$s\,\,\,$}
&\makebox[3em]{0}&\makebox[3em]{1}&\makebox[3em]{2}
&\makebox[3em]{3}&\makebox[3em]{4}\\\hline\hline
$2$ &1&0&-&-&-\\\hline
$4$ &2&2&0&0&-\\\hline
$8$ &4&4&4&4&0\\\hline
$16$ &8&6&5&5&4\\\hline
$32$ &9&8&6&6&5\\\hline
\end{tabular}
\vspace{15pt}
\caption{Maximal values of $t$}
 \end{table}
If an $H$-type algebra corresponds to signature $(s^*,t^*)$, then there exist all $H$-type algebras with the signatures $(0,t^*)$, \dots, $(s^*-1,t^*)$; $(s^*,0)$,\dots, $(s^*,t^*-1)$. The space $\mathfrak h$ of these $H$-type algebras can be the minimal dimensional admissible representation of the corresponding Clifford algebra, but can be of a bigger dimension. 
It follows from Tables~1 and 2 that, for example, given the dimension $4$ of  $\mathfrak{h}$ the following $H$-type algebras exist with the signatures of $\mathfrak{z}$ equal to
\[
(0,2),\quad (1,2),\quad(2,0),\quad(3,0);\qquad(3,0),\quad(1,1),\quad(1,2),
\]
as well as with the signatures $(0,1)$ and $(1,0)$, but do not exist with the signatures $(2,1)$ and $(0,3)$.

The metric $(\cdot\, ,\cdot)_Z$ on $N$, see Section 3.1,  is now normalized so that $(e_i\times e_j, e_i\times e_j)=\varepsilon_{ij}(s,t)/k$, where
\[
\varepsilon_{ij}(s,t)=
\begin{cases} 1, &  \text{for $(i,j)\in \mathfrak p_{\ell}''$, $\ell=1,\dots, s$;} \\
-1, & \text{for $(i,j)\in \mathfrak p_{\ell}''$, $\ell=s+1,\dots, s+t$.} 
\end{cases}
\]

\begin{remark}\label{rem4}
For the space $\mathfrak{h}=U$ of dimension  $u2^r$, where $u$ is odd, Theorems~\ref{te6} and~\ref{te7} are applied as for the dimension $2^r$.
\end{remark}

Analogously to Theorem~\ref{th.rational} we prove the following statement.
\begin{theorem}\label{th.rational_indef}
The pseudo $H$-type algebras admit rational structure constants.
\end{theorem}

\subsubsection{Existence of solution and isomorphisms}
Let us analyse the equation $E_1E_2+E_2E_1=0$. If the matrix $E_1E_2+E_2E_1$ has a non-zero element at a place $(i,j)$, then it is of the form $\alpha_{ik_1}\alpha_{k_1j}+\alpha_{ik_2}\alpha_{k_2j}$, which is the
result of the product and sum of the elements at the places $(i,k_1)$, $(k_1,i)$, $(k_2,j)$, and $(j,k_2)$ in $E_1$, and  $(i,k_2)$, $(k_2,i)$, $(k_1,j)$, and $(j,k_1)$ in $E_2$. Without loss of generality
we assume that $i<k_1<k_2<j$. Otherwise, we consider the upper triangles in the matrices $E_1$ and $E_2$. The difference between symmetric and skew-symmetric matrices is
the change of sign at the elements $\alpha_{ik_1}$ and $\alpha_{k_2j}$ or at the elements $\alpha_{ik_2}$ and $\alpha_{k_1j}$ which does not influence the equation $\alpha_{ik_1}\alpha_{k_1j}+\alpha_{ik_2}\alpha_{k_2j}=0$. Therefore, the existence of solutions for $\alpha_k$ is the same as for the positive definite metric. Indeed, we split this problem in two.
Given a set $\mathfrak p_{\ell}''$ of pairs $(i,j)$, and therefore,  a skew-symmetric matrix $E_{\ell}=\{\alpha^{(\ell)}_{ij}\}$ with $\alpha^{(\ell)}_{ij}=\pm 1$ for all $(i,j)\in \mathfrak p_{\ell}''$,  $\alpha^{(\ell)}_{ij}=0$ for all $(i,j)\not\in \mathfrak p_{\ell}''$, and $\alpha^{(\ell)}_{ji}:=-\alpha^{(\ell)}_{ij}$,  let us introduce the diagonal matrix $B_{\ell}$ having 1 at $(i,i)$ place and -1 at $(j,j)$ place, $(i,j)\in \mathfrak p_{\ell}''$, $i<j$, and zeros otherwise. Then $B_{\ell}^2=I$, and $E_{\ell}B_{\ell}$ is a symmetric matrix which is obtained from $E_{\ell}$ by multiplication by -1 of the lower triangle part.
Now    
\begin{itemize}
\item  We are looking for the
orthogonal matrices $E_{\ell}$, which are the solutions  to the equations $E_j^2=-I$ and $E_iE_j=-E_jE_i$ for $(i,j)\in \mathfrak p_{\ell}''$, $\ell=1,\dots, s+t$ as for the positive definite metric;
\item Change  $E_{\ell}\to \tilde{E}_{\ell}:=E_{\ell}B_{\ell}$ for $\ell=s+1,\dots, s+t$.
\end{itemize}
Then  $E_iE_j=-E_jE_i$,  $E_i\tilde{E}_j=-\tilde{E}_jE_i$, $\tilde{E}_i\tilde{E}_j=-\tilde{E}_j\tilde{E}_i$, $E_j^2=-I$, and $\tilde{E}_j^2=I$.
 This gives the existence of solutions for  $\alpha_{ij}$ and the existence
of isomorphism of the resulting pseudo $H$-type algebras for different solutions.

\section{Applications to combinatorial and orthogonal designs}

\subsection{Square 1-factorization of a complete graph}

Given a set of partitions  $\mathfrak{p}$ of $2k$ numbers into pairs we want to check all possible subset of partitions $\mathfrak{p}'\subseteq \mathfrak{p}$ in which we look for $\mathfrak{p}''\subseteq \mathfrak{p}'\subseteq \mathfrak{p}$  
of maximal possible length satisfying the following condition: any two partitions $\mathfrak{p}_{{m_{\alpha}}}\in \mathfrak{p}''$ and $\mathfrak{p}_{{m_{\beta}}}\in\mathfrak{p}''$ set together give disjoint cycles of length 4. This can be reformulated in terms of graphs as in Section~3.1. Given a complete graph with $2k$ vertices at the numbers $1\dots, 2k$ we look at the maximal
number of 1-factors from a 1-factorization defined by $\mathfrak{p}'$
satisfying the property:  any two 1-factors  set together give disjoint cycles of length 4.
In particular, the Steiner tournament  for a complete graph $K_8$ of $8$ vertices guarantees the existence  of maximal $\mathfrak{p}''= \mathfrak{p}'$ of length $p=7$. 

The general case of higher dimensions represents a hard problem of combinatorial design related to so-called {\it 1-factorization conjecture}, i.e., the problem of counting the number of non-isomorphic 1-factorizations of a complete graph $K_{2n}$. For small $n$
the answer is known only for $n\leq 7$. The unique factorization exists for $n=1,2,3$. In the case $n=4$, there are 6 non-isomorphic factorizations, two of which,  Kirkman and Steiner were mentioned before. Other non-trivial non-isomorphic factorizations are known for $n=5,6,7$, see \cite{Kaski}. 
The existence of $\mathfrak{p}''$ of maximal length is achieved for $K_{2k}$ if  $k=2^n$.  

\begin{theorem}\cite[Kobayashi and Nakamura]{KN}
There exists a 1-factorization of a complete graph $K_{2k}$ such that  any two 1-factors set together  form the union of disjoint cycles of length 4, if and only if, $k=2^n$, $n\geq 1$.
\end{theorem}

As an application to combinatorial design, let us formulate the following theorem.

\begin{theorem}\label{tcd}
There exists a 1-factorization of a complete graph $K_{2k}$ such that at most  $\rho(2k)-1$  one-factors satisfy the condition: being set together pairwise they  form the union of disjoint cycles of length 4.
\end{theorem}

\noindent
\begin{remark} In general $\rho(2^n)\leq 2^n$ with the equality only for $n=0,1,2,$ and $3$.
\end{remark}

\begin{proof}
In order to prove Theorem~\ref{tcd} let us start with a statement on interplay between the properties of the operator $J$ defined by \eqref{p1} and the cycle properties of the partitions $\mathfrak{p}'$.

Let the vectors $\omega_m=\sum_{(i,j)\in \mathfrak{p}_{l_m}}(e_i\times e_j)\}$, $\mathfrak{p}_{l_m}\in \mathfrak{p}'$ for  $m=m_1,\dots,m_p$, for some $1<p\leq 2k-1$, and let
$J_{\omega_{m_{\alpha}}}J_{\omega_{m_{\beta}}}=-J_{\omega_{m_{\beta}}}J_{\omega_{m_{\alpha}}}$ for every $\alpha,\beta\in \{1,\dots, p\}$ and $\alpha\neq \beta$. 
 Then  any two 1-factors $\mathfrak{p}_{\omega_{m_{\alpha}}}$ and $\mathfrak{p}_{\omega_{m_{\beta}}}$ set together give disjoint cycles of length 4.

Indeed, by the condition on $J_{\omega_{m_{\alpha}}}$ and $J_{\omega_{m_{\beta}}}$, and from the fact that  both operators act as signed permutations of the basis of $U$, we have
\[
\begin{array}{lcccr}
e_k&\stackrel{J_{\omega_{m_{\alpha}}}}\longrightarrow & \beta_1 e_s &\stackrel{J_{\omega_{m_{\beta}}}}\longrightarrow& \pm e_{\ell},\\
e_k&\stackrel{J_{\omega_{m_{\beta}}}}\longrightarrow & \beta_2 e_r &\stackrel{J_{\omega_{m_{\alpha}}}}\longrightarrow& \mp e_{\ell},
\end{array}
\]
where $\beta_1, \beta_2$ are +1 or -1.  Thus the vertices $k,\ell,s,r$ together with edges $ks$, $kr$, $sl$, $rl$ define a separate cycle in the  graph $\mathfrak{p}_{m_{\alpha}}\cup \mathfrak{p}_{m_{\beta}}$. The statement of Theorem~\ref{tcd} follows from Theorem~\ref{t6}  with $J_{w_{m_j}}$, where $\omega_{m_1},\ldots,\omega_{m_p}$ is an orthonormal basis of the center~$\mathfrak z$.
\end{proof}

\begin{remark} It follows from the proof, that if the center of a pseudo $H$-type algebra is of dimension $>1$, then the dimension of the horizontal space is a multiple of~4.
\end{remark}

\subsection{Space-time block codes, orthogonal designs, and  wireless communication}

Surprisingly, the above results have applications in wireless networks with multiple transmit  antennas, in which encoded signals undergo fading under local scattering, 
intersymbol interference under multipath propagation, the channels are time varying due to mobile motions, and co-channel interference due to cellular spectrum reuse. 
However, we use pocket communicators which must remain relatively simple commercially accessible and reliable in different type of environments.
The idea boiled down to {\it space-time block coding} for communication over Rayleigh fading channels. Data is encoded using space-time block coding and then 
splits into $n$ streams, which are then simultaneously transmitted via $n$ transmit antennas. The first two-branch transmit diversity scheme was proposed by Alamouti~\cite{Alamouti}
in 1998, and then, this scheme was generalized to $n$-brach scheme by Tarokh, Jafarkhani, and Calderbank~\cite{TJC}. The proposed construction was inductive, i.e., 
the bigger size block matrices were constructed using block matrices of smaller size. The direct method was proposed by Morier-Genoud and Ovsienko~\cite{MGO} in 2013.

Transmit diversity is provided by the simplest {\it linear processing orthogonal designs}  preserving the following properties
\begin{itemize}
\item No loss in bandwidth, i.e., maximum possible transmission rate at full diversity;
\item Maximum likelihood decoding algorithm at the receiver.
\end{itemize}

A model of transmitting process with space-time block coding can be formulated as follows. We want to transmit $n$ signals from $n$ transmitting antennas during the time
instances $t_1,\dots, t_T$ achieving maximal likelihood. Without loss of generality we assume that the receiving antenna is only one. Otherwise, the  maximum ratio combining
is applied, see e.g., \cite{Jaf}. Let us assume that the path gains from transmit antennas  to the receive antenna are $\mathbf{c}=(c_1,\dots, c_T)$, respectively. Then,  
the decoder receives signals $\mathbf{r}^j=(r^j_1,\dots, r^j_n)^T$ at each time $t_j$, $j=1,\dots, T$ as
\[
\mathbf{r}^j=\mathbf{c}W_j+\eta^j,
\] 
where $\eta^j$ is the Gaussian noise sample vector of the receive antenna and $W_j$ is a space time codeword, i.e., a $n\times T$ matrix with the entries $x_1,\dots,x_n$. Assume at this moment that $T=n$.
In order to satisfy the above property of full diversity and maximal likelihood, the matrices $W_j$ must form a real orthogonal design, see e.g., \cite[Section 4.4]{Jaf}. 

Orthogonal designs were introduced and studied in the 70's in a series of papers by Geramita and Seberry Wallis, summarised then in the monograph \cite{Geramita79}. 
They received much attention in relation with applications in wireless networks with multiple transmit  antennas. The idea boiled down to {\it space-time block coding} for communication over Rayleigh fading channels. Data is encoded using space-time block coding and then 
splits into $n$ streams, which are then simultaneously transmitted via $n$ transmit antennas, see e.g.,~\cite{Alamouti, MGO, TJC}.
For an account of orthogonal designs we refer to a comprehensive monograph~\cite{Geramita79}.

\begin{definition}
A linear processing  real  orthogonal design  in order $n$ of type $u_1,\dots, u_n$ on commuting variables $x_1,\dots, x_n$ is an $n \times n$ matrix $W$ with the entries \linebreak $\{0,\pm x_1,\dots, \pm x_n\}$, such that
\[
WW^T=\left(\sum_{k=1}^nu_k x_k^2\right)I,
\]
where $I$ is the  $n \times n$ unit matrix.
\end{definition}
A linear processing real orthogonal design  in order $n$ exists if and only if there exists  linear processing orthogonal design $\mathcal L$ in order $n$, such that 
\[
\mathcal L\mathcal L^T= \mathcal L^T\mathcal L=\left(\sum_{k=1}^n x_k^2\right)I,
\] 
see \cite[Theorem 3.4.1]{TJC}. Let
\[
\mathcal L=\sum_{k=1}^nx_kX_k,
\]
where the matrices $X_k$ are such that
\begin{itemize}
\item all entries are $\{0,\pm 1\}$;
\item the Hadamard (entrywise) product $X_k\circ X_j=0$ whenever $j\neq k$;
\item $X_kX_k^T=I$;
\item $X_kX_j^T+X_jX_k^T=0$.
\end{itemize}
Denote $E_k=X_1^TX_k$, $k=1,\dots, n$. Then, $E_1=I$, and $E_2,\dots,E_n$ form the HR family. The Hurwitz-Radon-Eckmann Theorem \ref{Eckmann} implies that the real orthogonal
design exists if and only if $n=2^r$ and $r=1,2,3$.

However, we may ask a question: how many non-zero entries $x_1,\dots, x_m$ one can transmit by making use of a linear processing real orthogonal design in order $n$?
A corollary from Theorem~\ref{t6} is the following result proved by Geramita and Seberry Wallis. 

\begin{theorem}[\cite{Geramita79} Corollary 1.4]\label{prop10}
A linear processing real orthogonal design in order $n=2k$ exists for at  most $m=\rho(n)$ non-zero entries. 
\end{theorem}

\begin{definition}
Two orthogonal designs $X$ and $Y$ of the same order are called amicable if $XY^T=YX^T$.
\end{definition}

Suppose that $X=\sum_{k=1}^sx_kX_k$ and $Y=\sum_{k=1}^ty_kY_k$ are two amicable orthogonal designs of order $n$ over variables $x_1,\dots, x_s$ and $y_1,\dots,y_t$ of type $u_1,\dots,u_s$ and $v_1,\dots,v_t$ respectively. So in addition to the matrices $X_k$, $k=1,\dots, s$ we operate with matrices $Y_j$, $j=1,\dots t$, such that
\begin{itemize}
\item all entries are $\{0,\pm 1\}$;
\item the Hadamard  product $Y_k\circ Y_j=0$ whenever $j\neq k$;
\item $Y_kY_k^T=-I_n$;
\item $Y_kY_j^T+Y_jY_k^T=0$;
\item $X_kY_j^T=Y_jX_k^T$.
\end{itemize}
Denote $E_k=X_1^TX_k$, $k=1,\dots, s$ and $E_k=Y_1^TY_k$, $k=s+2,\dots, s+t$. Then, the matrices $E_2,\dots,E_s,E_{s+2},\dots,E_{s+t}$ form the HR$(s-1,t-1)$ family.

Analogously to Theorem~\ref{prop10},   Wolfe \cite{Wolfe} proved the existence of amicable designs in
terms of a generalization of the Hurwitz-Radon function based on irreducible representation of Clifford algebras. The following question is natural to ask. Working with representations of Clifford algebras why we use the usual definition of orthogonality $E^TE=I$?
Our representation must act on the module as an isometry or anti-isometry with respect to a non-degenerate metric  $\eta$: 
\[
(E_ju,E_jv)=(u,v), \quad \text{for $j=2,\dots, s$;}\quad (E_ju,E_jv)=-(u,v), \quad \text{for $j=s+1,\dots, t$;}
\]
or equivalently
\[
 \eta E_j^T \eta=E_j^{-1},  \quad \text{for $j=2,\dots, s$;}\quad  \eta E_j^T \eta=-E_j^{-1}, \quad \text{for $j=s+1,\dots, t$}.
\]
Theorem~\ref{te6} and~\ref{te7} imply the following corollary related to Problem~5.17 from \cite[Section 5.3]{Geramita79} for even order  
amicable $\eta$-orthogonal designs.
 In particular, Tables~1 and~2
show  the number of  variables  in amicable designs of order $2k$, and when these designs do not exist  for some small dimensions. If order is odd, then $s,t\leq 1$, see \cite{Geramita79},
and this case is trivial.

\begin{theorem}\label{prop11} Given non-zero variables $y_1,\dots,y_t$ of an amicable $\eta$-orthogonal design of order $2k$, Theorem~\ref{te6} and Remark~\ref{rem4} give the maximal number of variables $x_1,\dots, x_s$, $s\leq \rho_t(2^{r})-m(t,r)+1$, where $k=u2^{r-1}$ and $u$ is odd. Analogously, given non-zero variables $x_1,\dots, x_s$  of an amicable $\eta$-orthogonal design of order $2n$, Theorem~\ref{te7} gives the maximal number of variables  $y_1,\dots,y_t$.
\end{theorem}
\begin{proof}
Assume that $e_1,\dots, e_{2k}$ is the basis of a vector space $U$ and $y_1,\dots,y_t$ are variables of the given design $Y$ and $x_1,\dots, x_s$ are non-zero entries of the amicable design $X$. Then the matrices $X_1,\dots, X_s$ and $Y_1,\dots, Y_t$ of dimension
 $(2k\times 2k)$ defined above act on $U$. The matrices  $E_2,\dots,E_s,E_{s+1},\dots,E_{s+t}$ form the HR$(s-1,t)$ that form the HR$(s-1,t)$ family and represent the basis of the center of a pseudo $H$-type algebra. Theorem~\ref{te6}
 implies that $s\leq \rho_t(2^{r})-m(r,t)+1$ and the upper bound is achieved. Analogously, the second part of the theorem follows from Theorem~\ref{te7}.
\end{proof}

\end{document}